\documentclass[twoside,centertags                  
]{amsart}

\usepackage{tikz}
\usetikzlibrary{matrix,arrows,calc,decorations.pathmorphing,fit,shapes.geometric,decorations.pathreplacing}

\usepackage[cmtip,color,all,curve,
arc,poly
]{xy}

\usepackage{amssymb}
\usepackage{mathrsfs}

\newtheorem{thm}[equation]{Theorem}
\newtheorem{cor}[equation]{Corollary}
\newtheorem{lem}[equation]{Lemma}
\newtheorem{prop}[equation]{Proposition}
\theoremstyle{definition}
\newtheorem{defn}[equation]{Definition}
\theoremstyle{remark}
\newtheorem{rem}[equation]{Remark}
\newtheorem{exm}[equation]{Example}

\newcommand{\abs}[1]{|#1|}

\def\r{\rightarrow} 

\newcommand{\id}[1]{\mathrm{id}_{#1}}

\def\hom{\operatorname{Hom}}

\def\st{\stackrel} 

\def\To{\longrightarrow}

\def\colim{\mathop{\operatorname{colim}}}

\newcommand{\norm}[1]{\parallel\!#1\!\parallel}

\numberwithin{equation}{section}

\newdir{ >}{{}*!/-3pt/@{>}}
\newdir{> }{{}*!/10pt/@{>}}
\newdir{>> }{{}*!/10pt/@{>>}}

\begin{document}

\title[Cylinders for non-symmetric  DG-operads]
{Cylinders for non-symmetric  DG-operads via homological perturbation theory}%
\author{Fernando Muro}%
\address{Universidad de Sevilla,
Facultad de Matem\'aticas,
Departamento de \'Algebra,
Avda. Reina Mercedes s/n,
41012 Sevilla, Spain}
\email{fmuro@us.es}
\urladdr{http://personal.us.es/fmuro}


\begin{abstract}
We construct small cylinders for cellular non-symmetric DG-oper\-ads over an arbitrary commutative ring by using the basic perturbation lemma from homological algebra. We show that our construction, applied to the $A$-infinity operad, yields the operad parametrizing $A$-infinity maps whose linear part is the identity. We also compute some other examples with non-trivial operations in arities $1$ and $0$. 
\end{abstract}

\maketitle


\section*{Introduction}

Cylinders are basic tools to do homotopy theory in any context. The existence of cylinders is guaranteed by the axioms of model categories. 
In cofibrantly generated model categories, cylinders can be constructed out of generating cofibrations, but they are huge, not useful for explicit computations. In many specific model categories, there are nice and small cylinders for cofibrant objects, e.g.~topological spaces (the product with the interval), chain complexes (well known), differential graded (DG) algebras, and connected commutative DG-algebras in characteristic $0$, see \cite[\S1]{ah}.

Any DG-algebra has a cofibrant resolution of the form $(T(V),d+\partial)$, where $(T(V),d)$ is the free DG-algebra on a chain complex $V$ and $\partial$ is a perturbation of the differential $d$. The cylinder  of such a DG-algebra is $(T(IV),d+\partial_{I})$, where $(T(IV),d)$ is the free DG-algebra on the cylinder $IV$ of the chain complex $V$ and $\partial_{I}$ is a perturbation defined from $\partial$ in a straightforward way, using a chain homotopy compatible with the associative algebra structure. The commutative case starts similarly, with a perturbed free commutative DG-algebra, but quickly diverges. The differential is defined in terms of locally nilpotent derivations and formal exponentials, where factorial denominators appear.

DG-operads are closely related to DG-algebras, the arity $1$ part $\mathcal O(1)$ of a DG-operad $\mathcal O$ is a DG-algebra. However, one of the DG-operad laws contains a switch,
\begin{align*}
(x\circ_iy)\circ_jz&=(-1)^{\abs{y}\abs{z}}(x\circ_jz)\circ_{i+n-1}y,\quad j<i,\quad z\in\mathcal O(n).
\end{align*}
Therefore, in any possible construction of cylinders for DG-operads we can expect some of the complications of the commutative DG-algebra case.
We  work with non-symmetric DG-operads, which avoids further switches associated to symmetric group actions. 
Unlike in the symmetric case, the category of non-symmetric DG-operads is endowed with a model structure over any commutative ground ring \cite{htnso, huainfa}. 

Any DG-operad has a cofibrant resolution of the form $(\mathcal F(V),d+\partial)$, where $(\mathcal F(V),d)$ is the free DG-operad on a sequence of chain complexes $V=\{V(n)\}_{n\geq 0}$ and $\partial$ is a perturbation of the differential $d$. We construct a cylinder for such a DG-operad of the form $(\mathcal F(IV),d+\partial_{I})$, where  $(\mathcal F(IV),d)$ is the free DG-operad on the sequence of chain cylinders $IV=\{IV(n)\}_{n\geq 0}$ and $\partial_{I}$ is a perturbation defined from $\partial$ by using tools from homological perturbation theory. The definition of $\partial_{I}$ is recursive. We compute some examples in Sections \ref{examples} and \ref{linear}. The most remarkable one is the  $A$-infinity operad $\mathcal A_{\infty}$, which has the previous form. Maps from $\mathcal A_{\infty}$ to an endomorphism operad correspond to $A$-infinity algebra structures, and homotopies with respect to our cylinder correspond to $A$-infinity morphisms whose linear part is the identity. 

Fresse \cite{occcomaoo} defined cylinders for symmetric DG-operads arising as the cobar construction of a coaugmented connected DG-cooperad. The $A$-infinity operad $\mathcal A_{\infty}$ arises in this way, it is the cobar construction of the Koszul dual cooperad of the associative operad. Fresse's cylinder of $\mathcal A_{\infty}$ coincides with ours, modulo symmetrization. His formulas are closed, not recursive, and work for symmetric operads. Our construction does not rule out operads with non-trivial operations in arities $0$ and $1$, such as the unital $A$-infinity operad \cite{ckdt,huainfa,uass} or the DG-operad for homotopy associative algebras with derivation \cite{oaad}, considered in Example \ref{conder}, and works straightaway in the relative case. We actually describe the cylinder of a DG-operad concentrated in arities $0$ and $1$, observing that our construction generalizes the classical cylinder of DG-algebras. In the final section, we consider a family of DG-operads, called linear, where our formulas greatly simplify. This family contains interesting examples, mainly in the relative case.

\subsection*{Acknowledgments}

The author was partially supported
by the Andalusian Ministry of Economy, Innovation and Science under the grant FQM-5713 and by the Spanish Ministry of Economy under the MEC-FEDER grant MTM2013-42178-P.

\section{Cylinders}

We work over an arbitrary commutative ground ring $\Bbbk$. The symmetry constraint in the monoidal categories of ($\mathbb Z$-)graded ($\Bbbk$-)modules and chain complexes uses the Koszul sign rule. Differentials have degree $-1$. As a graded module, the suspension $\Sigma X$ of a chain complex $X$ is $(\Sigma X)_n=X_{n-1}$ with differential $d_{\Sigma X}=-d_X$. We denote by $\sigma\colon X\r\Sigma X$ the degree $+1$ isomorphism that is the identity degreewise, which satisfies $d_{\Sigma X}\sigma =-\sigma d_X$. Therefore, given $x\in X_{n-1}$ we often write $\sigma(x)$ for $x$ itself regarded as an element in $(\Sigma X)_n$.

\begin{defn}\label{primen}
	A \emph{strong deformation retraction}, or simply an \emph{SDR}, consists of two chain complexes $X$ and $Y$ and a diagram 
	$$\xymatrix@C=40pt{X\ar@<.5ex>[r]^-{i}&*+[l]{Y}\ar@<.5ex>[l]^-p\ar@(ru,rd)^ h}$$
	where $i$ and $p$ are chain maps, $h$ is a chain homotopy from $ip$ to the identity in $Y$, i.e.~a degree $+1$ map satisfying
	$$ip-1_Y=dh+hd,$$
	and the following equations are satisfied,
	\begin{align*}
	pi&=1_X,&
	ph&=0,&
	hi&=0,&
	h^2&=0.
	\end{align*}
	
	A \emph{strong pseudo-cylinder} of a chain complex $X$ is a diagram
	$$\xymatrix{X
		\ar@<.5ex>[r]^-{i_0}
		\ar@<-.5ex>[r]_-{i_1}
		&
		C\ar[r]^-{p}\ar@(ul,ur)^ h
		&
		X
	}$$
	where $C$ is a chain complex, $i_0$, $i_1$ and $p$ are chain maps, $h$ is a chain homotopy such that
	$$\xymatrix@C=40pt{X\ar@<.5ex>[r]^-{i_0}&*+[l]{C}\ar@<.5ex>[l]^-p\ar@(ru,rd)^ h}$$
	is an SDR, and $pi_1=1_X$. We call it a \emph{strong cylinder} if $(i_0,i_1)\colon X\oplus X\r C$ is a cofibration in the projective model structure on chain complexes \cite[Theorem 2.3.11]{hmc}, and hence the underlying diagram obtained by forgetting $h$ is a cylinder object factorization in the model theoretic sense.
	
	The \emph{canonical strong pseudo-cylinder} of a chain complex $X$ is given by 
	$$\xymatrix{X
		\ar@<.5ex>[r]^-{i_0}
		\ar@<-.5ex>[r]_-{i_1}
		&
		IX\ar[r]^-{p}\ar@(ul,ur)^{h_I}
		&
		X
	}$$
	where $IX$, 
	as a graded module, is
	$$IX=X\oplus \Sigma X\oplus X.$$
	The differential of $IX$ is 
	$$d_{IX}=\left(\begin{smallmatrix}
	d_X&1&0\\0&-d_X&0\\0&-1&d_X
	\end{smallmatrix}\right)
	$$
	and the structure maps are 
	\begin{align*}
	i_0&=\left(\begin{smallmatrix}1\\0\\0\end{smallmatrix}
	\right),& 
	i_1&=\left(\begin{smallmatrix}0\\0\\1\end{smallmatrix}
	\right),& 
	p&=\left(\begin{smallmatrix}1&0&1\end{smallmatrix}
	\right),&
	h_I&=\left(\begin{smallmatrix}
	0&0&0\\0&0&\sigma\\0&0&0
	\end{smallmatrix}\right). 
	\end{align*}
\end{defn}

The canonical strong pseudo-cylinder is a strong cylinder if and only if $X$ is a cofibrant chain complex. 

\begin{defn}
In this paper, all operads are non-symmetric. Hence, a graded operad or DG-operad $\mathcal O$ is a sequence of objects $\{\mathcal O(n)\}_{n\geq 0}$ equipped with structure maps
\begin{align*}
	\mathcal O(n)\otimes\mathcal O(p_1)\otimes\cdots \otimes\mathcal O(p_n)&\To \mathcal O(p_1+\cdots +p_n),\qquad n\geq 1,\quad p_1,\dots,p_n\geq 0,\\
	x_0\otimes x_1\otimes\cdots\otimes x_n&\;\mapsto\;x_0(x_1,\dots, x_n),
\end{align*}
satisfying 
\begin{align}\label{oprel}
&x_0(x_1,\dots,x_n)(y_1,\dots, y_{p_1+\cdots+p_n})=\\
\nonumber&\qquad\qquad\quad(-1)^{\epsilon}x_0(x_1(y_1,\dots,y_{p_1}),\dots,x_n(y_{p_1+\cdots+p_{n-1}+1},\dots,y_{p_1+\cdots+p_n})),
\end{align}
where the sign is simply determined by the Koszul rule
\begin{align*}
	\epsilon&=\sum_{i=2}^{n} \sum_{j=1}^{\sum\limits_{k=1}^{i-1}p_k}\abs{x_{i}}\abs{y_{j}},
\end{align*}
and an identity element $\id{}=\id{\mathcal O}\in\mathcal O(1)$ satisfying 
\begin{align*}
	\id{}(x)&=x=x(\id{},\dots,\id{}).
\end{align*}
We call $\mathcal O(n)$ the \emph{arity $n$} part of $\mathcal O$. In the case of DG-operads, the structure maps being chain maps translates in the \emph{operadic Leibniz rule},
\begin{align*}
	\,d(x_0(x_1,\dots,x_n))&=d(x_0)(x_1,\dots,x_n)+\sum_{i=1}^n(-1)^{\sum\limits_{j=0}^{i-1}\abs{x_j}} \!\!\!x_0(\dots,x_{i-1},d(x_i),x_{i+1},\dots),
\end{align*}
which implies $d(\id{})=0$. In $x_0(x_1,\dots, x_n)$ we usually omit those $x_i$, $1\leq i\leq n$, which are $x_i=\id{}$. Operads can also be definied in terms of \emph{operadic compositions},
\[\circ_i\colon \mathcal O(p)\otimes\mathcal O(q)\To\mathcal O(p+q-1),\quad 1\leq i\leq p,\;\; q\geq 0, \qquad 
x\circ_iy=x(\st{i-1}\dots,y,\st{p-i}\dots).
\]
In this case, the laws are
\begin{align}
\nonumber x\circ_i(y\circ_{j}z)&=(x\circ_iy)\circ_{i+j-1}z;\\
\label{operadlaw}(x\circ_iy)\circ_jz&=(-1)^{\abs{y}\abs{z}}(x\circ_jz)\circ_{i+n-1}y,\quad j<i,\quad z\in\mathcal O(n);\\
\nonumber \id{}\circ_1x&=x=x\circ_i\id{};
\end{align}
and, for DG-operads, the operadic Leibniz rule is equivalent to
\begin{align}\label{olr}
	d(x\circ_iy)&=d(x)\circ_iy+(-1)^{\abs{x}}x\circ_id(y).
\end{align}
\end{defn}

\begin{defn}\label{cilindrodgop}
	A \emph{strong pseudo-cylinder} of a DG-operad $\mathcal O$ is a sequence of strong pseudo-cylinders of chain complexes 
	$$\xymatrix{\mathcal O
		\ar@<.5ex>[r]^-{i_0}
		\ar@<-.5ex>[r]_-{i_1}
		&
		\mathcal P\ar[r]^-{p}\ar@(ul,ur)^ h
		&
		\mathcal O
	}$$
	such that $\mathcal P$ is a DG-operad and $i_0$, $i_1$ and $p$ are DG-operad maps. 
	We call it a \emph{strong cylinder} if $(i_0,i_1)\colon\mathcal O\amalg\mathcal O\r \mathcal P$ is a cofibration in the model structure on DG-operads transferred from the projective model structure on chain complexes, see \cite[Theorem 1.1]{htnso} or \cite[Proposition 1.8]{huainfa}. Abusing terminology, we sometimes say that $\mathcal P$ is a strong (pseudo-)cylinder of $\mathcal O$. We sometimes write $h=h_{\mathcal O}$ in order to avoid ambiguity.
\end{defn}

No compatibility condition with the DG-operad structure is required for the chain homotopy $h$. 
Many explicit constructions below will however satisfy $h(\id{})=0$. There seems to be no canonical strong pseudo-cylinders for arbitrary DG-operads. The aim of this paper is to construct nice strong pseudo-cylinders for a big class of quasi-free DG-operads, i.e.~DG-operads with free underlying  graded operad. Nice means that they are strong cylinders under some extra assumptions, like canonical strong pseudo-cylinders of chain complexes. As a toy example, we will start with honestly free DG-operads.

We consider planted planar trees with leaves, that we simply call \emph{trees}, and endow the set of inner vertices with the path order. Two vertices satisfy $v\leq w$ if, when we draw the paths from the root to $v$ and $w$, the path to $v$ bifurcates to the left or is contained in the path to $w$, e.g.~in the following tree the path order is indicated by the subscripts
\begin{equation}\label{tree}
\begin{array}{c}
\begin{tikzpicture}[level distance=7mm, sibling distance=11mm] 
\tikzstyle{every node}=[execute at begin node=$\scriptstyle,execute at end node=$]
\tikzstyle{level 3}=[sibling distance=4mm]
\tikzstyle{level 4}=[sibling distance=2mm]
\node {} [grow'=up]
	child{[fill] circle (2pt)
		child{[fill] circle (2pt)
			child{}
			child{[fill] circle (2pt) node [above] {v_3}}
			child{[fill] circle (2pt)
				child{}
				child{}
				child{}
				node [below right] {v_4}}
			node [below left] {v_2}}
		child{[fill] circle (2pt)
			child{}
			node [below right] {v_5}}
		node [below left] {v_1}};
\end{tikzpicture}
\end{array}
\end{equation}
see \cite[\S3]{htnso} for formal definitions. 
The \emph{arity} of a vertex $v$ is the number $\widetilde v$ of edges adjacent from above, e.g.~in the previous tree $\widetilde{v}_1=2$, $\widetilde{v}_2=3$, $\widetilde{v}_3=0$, $\widetilde{v}_4=3$, and $\widetilde{v}_5=1$. We just depict \emph{inner vertices}, i.e.~we do not draw the top vertices of the leaves or the bottom vertex of the root. An \emph{inner edge} is an edge which is neither a leaf nor the root, i.e.~such that the two adjacent vertices are inner vertices.

Given a sequence of graded modules or chain complexes $V=\{V(n)\}_{n\geq 0}$ and a tree $T$, we denote
\[V(T)=\bigotimes_{v} V(\tilde v).\]
This tensor product is indexed by the inner vertices of $T$, and it is taken in the path order. A tensor in $V(T)$ is usually denoted by labeling each inner vertex $v$ with an element in $V(\widetilde{v})$, e.g.
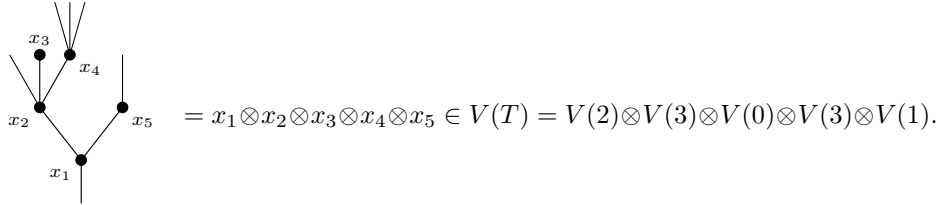
\begin{equation*} 
\begin{array}{c}
\begin{tikzpicture}[level distance=7mm, sibling distance=11mm] 
\tikzstyle{every node}=[execute at begin node=$\scriptstyle,execute at end node=$]
\tikzstyle{level 3}=[sibling distance=4mm]
\tikzstyle{level 4}=[sibling distance=2mm]
\node {} [grow'=up]
child{[fill] circle (2pt)
	child{[fill] circle (2pt)
		child{}
		child{[fill] circle (2pt) node [above] {x_3}}
		child{[fill] circle (2pt)
			child{}
			child{}
			child{}
			node [below right] {x_4}}
		node [below left] {x_2}}
	child{[fill] circle (2pt)
		child{}
		node [below right] {x_5}}
	node [below left] {x_1}};
\end{tikzpicture}
\end{array} = x_1\otimes x_2\otimes x_3\otimes x_4\otimes x_5\in V(T)=V(2)\otimes V(3)\otimes V(0)\otimes V(3)\otimes V(1).
\end{equation*}

Labeled trees are also used to denote iterated compositions in operads. There is only one bracketing compatible with the path order, e.g.~in the previous tree $x_1(x_2(-,x_3,x_4),x_5)$. This bracketing, that we call \emph{nested}, is the only one with no \[\dots)(\dots\] 
We can nest any bracketing by iterating \eqref{oprel}. As we will next see, the previous labeled tree can be regarded as an iterated composition in the free operad.

The underlying sequence of the free operad $\mathcal F(V)$ spanned by a sequence of graded modules or chain complexes $V$ is
$$\mathcal F(V)(n)=\bigoplus_TV(T),$$
see for instance \cite[\S5]{htnso}. 
This direct sum is indexed by the trees with $n$ leaves. 
The natural sequence of maps $V\r \mathcal F(V)$ is the inclusion of the direct summands indexed by corollas, $n\geq 0$,
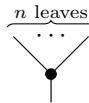
\[
\begin{tikzpicture}[level distance=5mm, sibling distance=5mm] 
\tikzset{every node/.style={ execute at begin node=$\scriptstyle ,%
		execute at end node=$, }}
\node {} [grow'=up] 
child { [fill] circle (2pt)
	child {{} node [above] (A) {}}
	child {node {\displaystyle\cdots}  edge from parent [draw=none] {}}
	child {{} node [above] (O) {}}
};
\draw [decoration={brace}, decorate] ($(A)+(-.08,0)$) -- node [above] {\text{\scriptsize $n$ leaves}} ($(O)+(.08,0)$) ;
\end{tikzpicture}
\]
The composition law $\circ_i$ is given by the symmetry isomorphisms $V(T)\otimes V(T')\cong V(T\circ_iT')$, where $T\circ_iT'$ is the tree obtained by grafting $T'$ in the $i^{\text{th}}$ leaf of $T$ (notice that the inner vertices of $T\circ_iT'$ are the disjoint union of the inner vertices of $T$ and $T'$). Using labeled trees, this is just grafting up to sign determined by the path order and the Koszul rule. The identity element is $\id{}=1\in V(|)=\Bbbk$.

\begin{lem}[{\cite[Lemma 2.3]{rwsobc}}]\label{tensorsdr}
	Given two SDRs
	$$\xymatrix{X\ar@<.5ex>[r]^-{i}&Y\ar@<.5ex>[l]^-p\ar@(ru,rd)^ h},\qquad
	\xymatrix{X'\ar@<.5ex>[r]^-{i'}&Y'\ar@<.5ex>[l]^-{p'}\ar@(ru,rd)^{h'}
	},$$
	the following diagram is also an SDR, that we call \emph{tensor product SDR},
	$$\xymatrix@C=40pt{X\otimes X'\ar@<.5ex>[r]^-{i\otimes i'}&*+[l]{Y\otimes Y'}\ar@<.5ex>[l]^-{p\otimes p'}\ar@(ru,rd)^{ h\otimes 1_{Y'}+ip \otimes h'}}.$$
\end{lem}

Tensor product SDRs behave well with respect to associativity constraints, so we can define tensor products of several SDRs by iteration without specifying a bracketing. However, this construction is not symmetric, so the order of tensor factors does matter. This is why we insisted on the path order in the definition of free operads. We could have equally worked with the chain homotopy $ h\otimes i'p'+1_Y\otimes h'$, but we must fix some convention, and we have decided to fix that in Lemma \ref{tensorsdr}.

Tensor products of strong pseudo-cylinders of chain complexes are defined by using tensor products of SDRs in the obvious way. The canonical strong pseudo-cylinder of a sequence of chain complexes is defined aritywise as in Definition \ref{primen}. The empty tensor product is the trivial strong pseudo-cylinder of $\Bbbk$, regarded as a chain complex concentrated in degree $0$,
$$\xymatrix{\Bbbk  
	\ar@<.5ex>[r]^-{1}
	\ar@<-.5ex>[r]_-{1}
	&
	\Bbbk  \ar[r]^-{1}\ar@(ul,ur)^{0}
	&
	\Bbbk.
}$$

\begin{defn}
	The \emph{canonical strong pseudo-cylinder} of a free DG-operad $\mathcal F(V)$ on a sequence $V$ of chain complexes
		$$\xymatrix{\mathcal F(V)
			\ar@<.5ex>[r]^-{\mathcal F(i_0)}
			\ar@<-.5ex>[r]_-{\mathcal F(i_1)}
			&
			\mathcal F(IV)\ar[r]^-{\mathcal F(p)}\ar@(ul,ur)^{h_{V}}
			&
			\mathcal F(V)
		}$$
	is defined on each $V(T)$ as the tensor product (in the path order) of the canonical strong pseudo-cylinders of the sequence of chain complexes $V$. 
\end{defn}

The empty tensor product appears when $T=|$, hence $h_V(\id{})=0$.

\begin{rem}\label{hv}
It is very easy to evaluate $h_{V}$ on labeled trees, e.g.
\begin{align*} 
&h_V\left(\begin{array}{c}
\begin{tikzpicture}[level distance=7mm, sibling distance=11mm] 
\tikzstyle{every node}=[execute at begin node=$\scriptstyle,execute at end node=$]
\tikzstyle{level 3}=[sibling distance=4mm]
\tikzstyle{level 4}=[sibling distance=2mm]
\node {} [grow'=up]
child{[fill] circle (2pt)
	child{[fill] circle (2pt)
		child{}
		child{[fill] circle (2pt) node [above] {x_3}}
		child{[fill] circle (2pt)
			child{}
			child{}
			child{}
			node [below right] {x_4}}
		node [below left] {x_2}}
	child{[fill] circle (2pt)
		child{}
		node [below right] {x_5}}
	node [below left] {x_1}};
\end{tikzpicture}
\end{array}\right)=
\begin{array}{c}
\begin{tikzpicture}[level distance=7mm, sibling distance=11mm] 
\tikzstyle{every node}=[execute at begin node=$\scriptstyle,execute at end node=$]
\tikzstyle{level 3}=[sibling distance=4mm]
\tikzstyle{level 4}=[sibling distance=2mm]
\node {} [grow'=up]
child{[fill] circle (2pt)
	child{[fill] circle (2pt)
		child{}
		child{[fill] circle (2pt) node [above] {x_3}}
		child{[fill] circle (2pt)
			child{}
			child{}
			child{}
			node [below right] {x_4}}
		node [below left] {x_2}}
	child{[fill] circle (2pt)
		child{}
		node [below right] {x_5}}
	node [below left] {h_I(x_1)}};
\end{tikzpicture}
\end{array}
+
(-1)^{\abs{x_1}}
\begin{array}{c}
\begin{tikzpicture}[level distance=7mm, sibling distance=11mm] 
\tikzstyle{every node}=[execute at begin node=$\scriptstyle,execute at end node=$]
\tikzstyle{level 3}=[sibling distance=4mm]
\tikzstyle{level 4}=[sibling distance=2mm]
\node {} [grow'=up]
child{[fill] circle (2pt)
	child{[fill] circle (2pt)
		child{}
		child{[fill] circle (2pt) node [above] {x_3}}
		child{[fill] circle (2pt)
			child{}
			child{}
			child{}
			node [below right] {x_4}}
		node [below left] {h_I(x_2)}}
	child{[fill] circle (2pt)
		child{}
		node [below right] {x_5}}
	node [below left] {i_0p(x_1)}};
\end{tikzpicture}\end{array}\\
&+(-1)^{\abs{x_1}+\abs{x_2}}
\begin{array}{c}
\begin{tikzpicture}[level distance=7mm, sibling distance=11mm] 
\tikzstyle{every node}=[execute at begin node=$\scriptstyle,execute at end node=$]
\tikzstyle{level 3}=[sibling distance=4mm]
\tikzstyle{level 4}=[sibling distance=2mm]
\node {} [grow'=up]
child{[fill] circle (2pt)
	child{[fill] circle (2pt)
		child{}
		child{[fill] circle (2pt) node [above] {h_I(x_3)\quad}}
		child{[fill] circle (2pt)
			child{}
			child{}
			child{}
			node [below right] {x_4}}
		node [below left] {i_0p(x_2)}}
	child{[fill] circle (2pt)
		child{}
		node [below right] {x_5}}
	node [below left] {i_0p(x_1)}};
\end{tikzpicture}
\end{array}
+(-1)^{\abs{x_1}+\abs{x_2}+\abs{x_3}}
\begin{array}{c}
\begin{tikzpicture}[level distance=7mm, sibling distance=13mm] 
\tikzstyle{every node}=[execute at begin node=$\scriptstyle,execute at end node=$]
\tikzstyle{level 3}=[sibling distance=4mm]
\tikzstyle{level 4}=[sibling distance=2mm]
\node {} [grow'=up]
child{[fill] circle (2pt)
	child{[fill] circle (2pt)
		child{}
		child{[fill] circle (2pt) node [above] {i_0p(x_3)\quad}}
		child{[fill] circle (2pt)
			child{}
			child{}
			child{}
			node [below] {\qquad  h_I(x_4)}}
		node [below left] {i_0p(x_2)}}
	child{[fill] circle (2pt)
		child{}
		node [below right] {x_5}}
	node [below left] {i_0p(x_1)}};
\end{tikzpicture}
\end{array}\\
&
+(-1)^{\abs{x_1}+\abs{x_2}+\abs{x_3}+\abs{x_4}}
\begin{array}{c}
\begin{tikzpicture}[level distance=7mm, sibling distance=14mm] 
\tikzstyle{every node}=[execute at begin node=$\scriptstyle,execute at end node=$]
\tikzstyle{level 3}=[sibling distance=4mm]
\tikzstyle{level 4}=[sibling distance=2mm]
\node {} [grow'=up]
child{[fill] circle (2pt)
	child{[fill] circle (2pt)
		child{}
		child{[fill] circle (2pt) node [above] {i_0p(x_3)\quad}}
		child{[fill] circle (2pt)
			child{}
			child{}
			child{}
			node [below] {\qquad\; i_0p(x_4)}}
		node [below left] {i_0p(x_2)}}
	child{[fill] circle (2pt)
		child{}
		node [below right] {h_I(x_5)}}
	node [below left] {i_0p(x_1)}};
\end{tikzpicture}
\end{array}.
\end{align*}

In terms of formulas, given $x_0\in IV(n)$  and $x_1,\dots, x_n\in \mathcal F(IV)$,
\begin{align}\label{hder}
	h_{V}(x_0(x_1,\dots,x_n))={}&h_I(x_0)(x_1,\dots,x_n)\\
\nonumber	&\hspace{-70pt}+
	\sum_{i=1}^{n}(-1)^{\sum\limits_{j=0}^{i-1}\abs{x_j}}i_0p(x_0)(i_0p(x_1),\dots,i_0p(x_{i-1}),h_{V}(x_i),x_{i+1},\dots,x_n).
\end{align}
This yields a straightforward way of computing $h_V$ on nested bracketings of elements in $V$. This equation need not hold if $x_0\in\mathcal F(IV)$ is an arbitrary element.
	
The canonical strong pseudo-cylinder of a free operad is a strong cylinder when $V$ is a sequence of cofibrant chain complexes.
\end{rem}

We now consider strong pseudo-cylinders on the coproduct of an arbitrary DG-operad and a free one.

Given two sequences of graded modules or chain complexes, $U$ and $V$, and a tree $T$, we define 
\[(U,V)(T)\cong\bigotimes_{u\text{ odd}} U(\widetilde{u})
\otimes\bigotimes_{v\text{ even}} V(\widetilde{v}).
\]
The first (resp.~second) tensor product is indexed by the inner vertices of odd (resp.~even) level, and the order of tensor factors should be the path order (in the right hand side of the formula we have separated odd and even inner vertices for lack of a better notation, the isomorphism is defined by the symmetry constraint). In \eqref{tree}, $v_1$ has level $1$, $v_2$ and $v_5$ have level $2$, and $v_3$ and $v_4$ have level~$3$. As above, a tensor in $(U,V)(T)$ is usually denoted by labeling each inner vertex $u$ of odd level with an element in $U(\widetilde{u})$
and each inner vertex $v$ of even level with an element in $V(\widetilde{v})$, e.g.
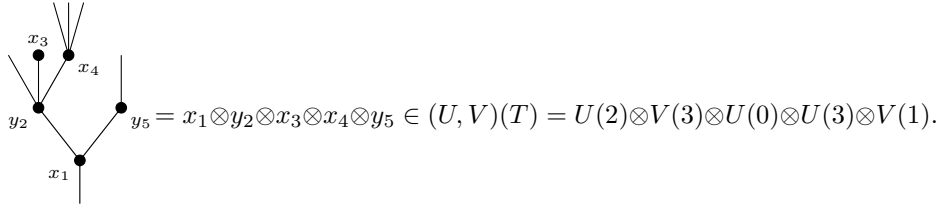
\begin{equation*} 
\begin{array}{c}
\begin{tikzpicture}[level distance=7mm, sibling distance=11mm] 
\tikzstyle{every node}=[execute at begin node=$\scriptstyle,execute at end node=$]
\tikzstyle{level 3}=[sibling distance=4mm]
\tikzstyle{level 4}=[sibling distance=2mm]
\node {} [grow'=up]
child{[fill] circle (2pt)
	child{[fill] circle (2pt)
		child{}
		child{[fill] circle (2pt) node [above] {x_3}}
		child{[fill] circle (2pt)
			child{}
			child{}
			child{}
			node [below right] {x_4}}
		node [below left] {y_2}}
	child{[fill] circle (2pt)
		child{}
		node [below right] {y_5}}
	node [below left] {x_1}};
\end{tikzpicture}\hspace{-10pt}
\end{array} = x_1\otimes y_2\otimes x_3\otimes x_4\otimes y_5\in (U,V)(T)=U(2)\otimes V(3)\otimes U(0)\otimes U(3)\otimes V(1).
\end{equation*}

Given an arbitrary operad $\mathcal O$ and a sequence of graded modules or chain complexes $V$, the sequence underlying the coproduct $\mathcal O\amalg\mathcal F(V)$ in the category of graded or DG-operads, see \cite[\S5]{htnso} or \cite[Remark 3.9]{udga}, is
\begin{equation}\label{coparb}
	(\mathcal O\amalg\mathcal F(V))(n)=\bigoplus_T(\mathcal O,V)(T).
\end{equation}
This direct sum is indexed by the trees with $n$ leaves, all of them in even level. In \eqref{tree}, the three topmost leaves have level $4$, and the two other leaves have level $3$, so this not an indexing tree here, but the following similar example is,
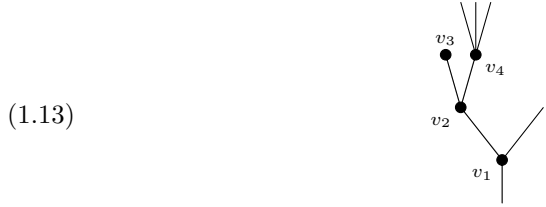
\begin{equation}\label{tree2}
\begin{array}{c}
\begin{tikzpicture}[level distance=7mm, sibling distance=11mm] 
\tikzstyle{every node}=[execute at begin node=$\scriptstyle,execute at end node=$]
\tikzstyle{level 3}=[sibling distance=4mm]
\tikzstyle{level 4}=[sibling distance=2mm]
\node {} [grow'=up]
child{[fill] circle (2pt)
	child{[fill] circle (2pt)
		child{[fill] circle (2pt) node [above] {v_3}}
		child{[fill] circle (2pt)
			child{}
			child{}
			child{}
			node [below right] {v_4}}
		node [below left] {v_2}}
	child{}
	node [below left] {v_1}};
\end{tikzpicture}
\end{array}
\end{equation}
Using (labeled) trees, structure maps are easily defined as follows. The inclusion of the first factor $\mathcal O\r \mathcal O\amalg\mathcal F(V)$ is the inclusion of the direct summands indexed by corollas. The canonical map $V\r \mathcal O\amalg\mathcal F(V)$ sends an element $y\in V(n)$ to the following labeled tree,
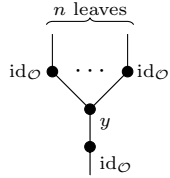
\[
	\begin{tikzpicture}[level distance=5mm, sibling distance=5mm] 
	\tikzset{every node/.style={ execute at begin node=$\scriptstyle ,%
			execute at end node=$, }}
	\node {} [grow'=up] 
	child{[fill] circle (2pt)
	child { [fill] circle (2pt)
		child{[fill] circle (2pt) child {{} node [above] (A) {}} node [left] {\id{\mathcal O}}}
				child {node {\displaystyle\cdots}  edge from parent [draw=none] {}}
		child{[fill] circle (2pt) child {{} node [above] (O) {}} node [right] {\id{\mathcal O}}}
	node [below right] {y}}
node [below right] {\id{\mathcal O}}};
	\draw [decoration={brace}, decorate] ($(A)+(-.08,0)$) -- node [above] {\text{\scriptsize $n$ leaves}} ($(O)+(.08,0)$) ;
	\end{tikzpicture}
\]
Composition $\circ_i$ is given by grafting into the $i^{\text{th}}$ leaf (taking into account the path order and the Koszul sign rule) and, if $u$ (resp.~$v$) is the bottom (resp.~top) vertex of the inner edge created by grafting and $v$ is in the $j^{\text{th}}$ edge adjacent to $u$ from above, then contracting this inner edge and labelling the vertex obtained by merging $u$ and $v$ with the element in $\mathcal O$ obtained by applying $\circ_j$ to the labels of $u$ and $v$ (this also involves a sign, according to the standing conventions, which is $-1$ up to the product of the degree of the label of $v$ by the sum of the degrees of the labels of the vertices strictly between $u$ and $v$), e.g.
\begin{equation*} 
\!\!\!\!\!\begin{array}{c}
\begin{tikzpicture}[level distance=7mm, sibling distance=4mm] 
\tikzstyle{every node}=[execute at begin node=$\scriptstyle,execute at end node=$]
\tikzstyle{level 3}=[sibling distance=4mm]
\tikzstyle{level 4}=[sibling distance=2mm]
\node {} [grow'=up]
	child{[fill] circle (2pt)
			child{[fill] circle (2pt) node [above] {y_2}}
			child{}
			child{}
			child{[fill] circle (2pt) node [above] {y_3}}
		node [below right] {x_1}};
\draw (1.2,1) node {\displaystyle \circ_2};
\tikzstyle{level 2}=[sibling distance=11mm]
\draw (2.4,0) node {} [grow'=up]
child{[fill] circle (2pt)
	child{[fill] circle (2pt)
		child{[fill] circle (2pt) node [above] {x_3'}}
		child{[fill] circle (2pt)
			child{}
			child{}
			child{}
			node [right] {x_4'}}
		node [below left] {y_2'}}
	child{}
	node [below left] {x_1'}};
\draw (3.4,1) node {\displaystyle =};
\draw (5.7,1.07) node {\displaystyle (-1)^{\abs{y_3}(\abs{x_1'}+\abs{y_2'}+\abs{x_3'}+\abs{x_4'})}};
\tikzstyle{level 2}=[sibling distance=7mm]
\tikzstyle{level 3}=[sibling distance=5mm]
\tikzstyle{level 4}=[sibling distance=4mm]
\tikzstyle{level 5}=[sibling distance=2mm]
\draw (8.9,0) node {} [grow'=up]
child{[fill] circle (2pt)
	child{[fill] circle (2pt) node [above] {y_2}}
	child{}
	child{[fill] circle (2pt)
			child{[fill] circle (2pt)
				child{[fill] circle (2pt) node [above] {x_3'}}
				child{[fill] circle (2pt)
					child{}
					child{}
					child{}
					node [right] {x_4'}}
				node [below left] {y_2'}}
			child{}
			node [right] {x_1'}}
	child{[fill] circle (2pt) node [above] {y_3}}
	node [below right] {x_1}};
\draw (3.4,-3) node {\displaystyle =};
\draw (6.3,-2.93) node {\displaystyle (-1)^{\abs{y_3}(\abs{x_1'}+\abs{y_2'}+\abs{x_3'}+\abs{x_4'})+\abs{y_2}\abs{x_1'}}};
\tikzstyle{level 2}=[sibling distance=7mm]
\tikzstyle{level 3}=[sibling distance=5mm]
\tikzstyle{level 4}=[sibling distance=4mm]
\tikzstyle{level 5}=[sibling distance=2mm]
\draw (10.4,-4) node {} [grow'=up]
child{[fill] circle (2pt)
	child{[fill] circle (2pt) node [above] {y_2}}
	child{}
		child{[fill] circle (2pt)
			child{[fill] circle (2pt) node [above] {x_3'}}
			child{[fill] circle (2pt)
				child{}
				child{}
				child{}
				node [right] {x_4'}}
			node [left] {y_2'}}
		child{}
	child{[fill] circle (2pt) node [above] {y_3}}
	node [below right] {x_1\circ_3 x_1'}};
\end{tikzpicture}
\end{array}
\end{equation*}

\begin{defn}\label{cispc}
	Given a DG-operad $\mathcal O$ equipped with a strong pseudo-cylinder
		$$\xymatrix{\mathcal O
			\ar@<.5ex>[r]^-{i_0}
			\ar@<-.5ex>[r]_-{i_1}
			&
			\mathcal P\ar[r]^-{p}\ar@(ul,ur)^{h_{\mathcal O}}
			&
			\mathcal O
		}$$
	and a sequence of chain complexes $V$, the \emph{canonical induced strong pseudo-cylinder} of the coproduct $\mathcal O\amalg\mathcal F(V)$ 
	$$\xymatrix@C=45pt{\mathcal O\amalg\mathcal F(V)
		\ar@<.5ex>[r]^-{i_0=i_0\amalg\mathcal F(i_0)}
		\ar@<-.5ex>[r]_-{i_1=i_1\amalg\mathcal F(i_1)}
		&
		\mathcal P\amalg\mathcal F(IV)\ar[r]^-{p=p\amalg\mathcal F(p)}\ar@(ul,ur)^{h_{\mathcal O, V}}
		&
		\mathcal O\amalg\mathcal F(V)
	}$$
	is defined on each $(\mathcal O,V)(T)$ as the tensor product of the strong pseudo-cylinder of $\mathcal O$ and the canonical strong pseudo-cylinder of the sequence of chain complexes $V$.
\end{defn}

If $\mathcal O$ is the initial DG-operad and we take the trivial strong pseudo-cylinder on it (Definition \ref{cspc}), with $\mathcal P=\mathcal O$, $i_0=i_1=p=1_{\mathcal O}$, and $h_{\mathcal O}=0$, then we recover the canonical strong pseudo-cylinder of $\mathcal F(V)$.

\begin{rem}\label{formulae}
	Computing $h_{\mathcal O,V}$ on labeled trees is as easy as with $h_V$ above, e.g.
	\begin{align*} 
	&h_{\mathcal O,V}\left(\begin{array}{c}
	\begin{tikzpicture}[level distance=7mm, sibling distance=11mm] 
	\tikzstyle{every node}=[execute at begin node=$\scriptstyle,execute at end node=$]
	\tikzstyle{level 3}=[sibling distance=4mm]
	\tikzstyle{level 4}=[sibling distance=2mm]
	\node {} [grow'=up]
	child{[fill] circle (2pt)
		child{[fill] circle (2pt)
			child{[fill] circle (2pt) node [above] {x_3}}
			child{[fill] circle (2pt)
				child{}
				child{}
				child{}
				node [below right] {x_4}}
			node [below left] {y_2}}
		child{}
		node [below left] {x_1}};
	\end{tikzpicture}
	\end{array}\right)=
	\begin{array}{c}
	\begin{tikzpicture}[level distance=7mm, sibling distance=11mm] 
	\tikzstyle{every node}=[execute at begin node=$\scriptstyle,execute at end node=$]
	\tikzstyle{level 3}=[sibling distance=4mm]
	\tikzstyle{level 4}=[sibling distance=2mm]
	\node {} [grow'=up]
	child{[fill] circle (2pt)
		child{[fill] circle (2pt)
			child{[fill] circle (2pt) node [above] {x_3}}
			child{[fill] circle (2pt)
				child{}
				child{}
				child{}
				node [below right] {x_4}}
			node [below left] {y_2}}
		child{}
		node [below left] {h_{\mathcal O}(x_1)}};
	\end{tikzpicture}
	\end{array}
	+
	(-1)^{\abs{x_1}}
	\begin{array}{c}
	\begin{tikzpicture}[level distance=7mm, sibling distance=11mm] 
	\tikzstyle{every node}=[execute at begin node=$\scriptstyle,execute at end node=$]
	\tikzstyle{level 3}=[sibling distance=4mm]
	\tikzstyle{level 4}=[sibling distance=2mm]
	\node {} [grow'=up]
	child{[fill] circle (2pt)
		child{[fill] circle (2pt)
			child{[fill] circle (2pt) node [above] {x_3}}
			child{[fill] circle (2pt)
				child{}
				child{}
				child{}
				node [below right] {x_4}}
			node [below left] {h_I(y_2)}}
		child{}
		node [below left] {i_0p(x_1)}};
	\end{tikzpicture}\end{array}\\
	&\qquad\qquad\qquad+(-1)^{\abs{x_1}+\abs{y_2}}
	\begin{array}{c}
	\begin{tikzpicture}[level distance=7mm, sibling distance=11mm] 
	\tikzstyle{every node}=[execute at begin node=$\scriptstyle,execute at end node=$]
	\tikzstyle{level 3}=[sibling distance=4mm]
	\tikzstyle{level 4}=[sibling distance=2mm]
	\node {} [grow'=up]
	child{[fill] circle (2pt)
		child{[fill] circle (2pt)
			child{[fill] circle (2pt) node [above] {h_{\mathcal O}(x_3)\quad}}
			child{[fill] circle (2pt)
				child{}
				child{}
				child{}
				node [below right] {x_4}}
			node [below left] {i_0p(y_2)}}
		child{}
		node [below left] {i_0p(x_1)}};
	\end{tikzpicture}
	\end{array}
	+(-1)^{\abs{x_1}+\abs{y_2}+\abs{x_3}}
	\begin{array}{c}
	\begin{tikzpicture}[level distance=7mm, sibling distance=11mm] 
	\tikzstyle{every node}=[execute at begin node=$\scriptstyle,execute at end node=$]
	\tikzstyle{level 3}=[sibling distance=4mm]
	\tikzstyle{level 4}=[sibling distance=2mm]
	\node {} [grow'=up]
	child{[fill] circle (2pt)
		child{[fill] circle (2pt)
			child{[fill] circle (2pt) node [above] {i_0p(x_3)\quad}}
			child{[fill] circle (2pt)
				child{}
				child{}
				child{}
				node [below right] {h_{\mathcal O}(x_4)}}
			node [below left] {i_0p(y_2)}}
		child{}
		node [below left] {i_0p(x_1)}};
	\end{tikzpicture}
	\end{array}
	\end{align*}


		For $x_0\in IV(n)$ and $x_1,\dots,x_n\in\mathcal P\amalg\mathcal F(IV)$, formula \eqref{hder} holds, replacing $h_V$ with $h_{\mathcal O,V}$. 
	
	Formula \eqref{hder} also holds if $x_0\in\mathcal P(n)$ and each $x_i\in\mathcal P\amalg\mathcal F(IV)$, $1\leq i\leq n$, is either $x_{i}=\id{}$ or $x_i=y_{i,0}(y_{i,1},\dots,y_{i,p_i})$ with $y_{i,0}\in IV(p_i)$,  replacing $h_V$ with $h_{\mathcal O,V}$ and $h_I$ with $h_{\mathcal O}$. 
	
	
\end{rem}

We now consider a class of DG-operads obtained by perturbing the differential on a coproduct of the form $\mathcal O \amalg\mathcal F(V)$, and define strong pseudo-cylinders for them.

\begin{defn}\label{tc}
	Given a DG-operad $\mathcal O$, a sequence of chain complexes $V$, and a sequence of degree $-1$ maps $\partial\colon V\r\mathcal O$ satisfying
	\[\partial d_V +d_\mathcal O\partial=0,\]
	the \emph{twisted coproduct} $\mathcal O \amalg_\partial\mathcal F(V)$ is the DG-operad with the same underlying graded operad as $\mathcal O \amalg\mathcal F(V)$ and differential $d_\partial$ satisfying
	\begin{align*}
		d_{\partial}(x)&=d_{\mathcal O}(x),\quad x\in \mathcal O;&
		d_{\partial}(y)&=d_V(y)+\partial(y),\quad y\in V.
	\end{align*}
\end{defn}

Using the structure of the coproduct with a free operad, recalled above, it is straightforward to check that the differential of the twisted coproduct is unique and well defined by the two last equations. The following is an alternative homotopical argument. The first equation is equivalent to say that $\partial$ is a sequence of chain maps from the desuspension $\Sigma^{-1}V\r\mathcal O$. This sequence of maps extends to an operad map $\mathcal F(\Sigma^{-1}V)\r\mathcal O$, and $\mathcal O \amalg_\partial\mathcal F(V)$ is the mapping cone of this operad map. For this, we use the cone on $\mathcal F(\Sigma^{-1}V)$ obtained as the free operad on the sequence of usual chain cone on $\Sigma^{-1}V$. These are model theoretic cones whenever $V$ is a sequence of cofibrant complexes. Hence, in that case, the canonical inclusion $\mathcal O \r \mathcal O \amalg_\partial\mathcal F(V)$ is a principal cofibration.

A DG-operad map $f\colon \mathcal O \amalg_\partial\mathcal F(V)\r\mathcal Q$ is the same as pair formed by a DG-operad map $f_{\mathcal O}\colon\mathcal O\r\mathcal Q$ and a sequence of maps of graded modules $f_V\colon V\r\mathcal Q$ such that $d_{\mathcal Q}f_V=f_Vd_V+f_{\mathcal O}\partial$. For $\partial=0$, the twisted coproduct is the plain coproduct.

In the construction of canonical induced strong pseudo-cylinders of twisted coproducts, we will use 
the following well-known lemma from homological algebra.

\begin{lem}[Basic Perturbation Lemma \cite{tezt}]\label{bpl}
	Given an SDR 
	$$\xymatrix@C=40pt{(X,d_X)\ar@<.5ex>[r]^-{i}&*+[l]{(Y,d_Y)}\ar@<.5ex>[l]^-p\ar@(ru,rd)^ h}$$
	and a degree $-1$ map $\partial\colon Y\r Y$, called \emph{perturbation}, such that $\partial^2+d_Y\partial+\partial d_Y=0$ and the infinite sum
	$\Sigma_\infty=\sum_{n\geq 0}(\partial h)^n\partial$
	is well defined, i.e.~almost all summands vanish when evaluated at a given $y\in Y$, then there is a new SDR
	$$\xymatrix@C=60pt{(X,d_X+p\Sigma_\infty i)\ar@<.5ex>[r]^-{i+ h\Sigma_{\infty}i}&*+[l]{(Y,d_Y+\partial)}\ar@<.5ex>[l]^-{p+p\Sigma_\infty  h}\ar@(ru,rd)^{h+ h\Sigma_\infty h}}.$$
\end{lem}

\begin{rem}\label{filtration}
	The vanishing condition is fulfilled if $Y$ is equipped with an exhaustive increasing filtration $$0=F_{-1}Y\subset F_0Y\subset \cdots\subset F_nY\subset F_{n+1}Y\subset\cdots \subset Y,\qquad Y=\bigcup_{n\geq 0}F_nY,$$
	such that for $n\geq 0$ $$\partial(F_{n}Y)\subset F_{n-1}Y,\qquad  h(F_nY)\subset F_nY.$$
	This implies that, if $y\in F_{n}Y$, then $\Sigma_{\infty}(y)=\sum_{j=0}^{n-1}(\partial h)^{j}\partial(y)$. The maps $i$ and $p$ often preserve the filtration, like $h$. 
\end{rem}

\begin{rem}\label{hrec}
	The new chain homotopy $h'=h+ h\Sigma_\infty h$ satisfies the following equation, \[h'=h+h'\partial h.\]
\end{rem}

\begin{rem}\label{formulae2}
The differential of the twisted coproduct $\mathcal O \amalg_\partial\mathcal F(V)$ is obtained by perturbing the differential of the honest coproduct $\mathcal O \amalg\mathcal F(V)$. The perturbation $d_\partial-d_{\mathcal O \amalg\mathcal F(V)}$, that we also call $\partial$, is the only degree $-1$ self-map of $\mathcal O \amalg\mathcal F(V)$ satisfying $\partial(\mathcal O)=0$, the operadic Leibniz rule, and restricting to the original $\partial\colon V \r\mathcal O$ on $V$. Applying $\partial$ to a labeled tree with $m$ inner vertices of even level, we obtain a sum of labeled trees with the same shape (with signs coming from the path order and the Koszul sign rule), one for each $1\leq i\leq m$,  where all labels are the same except for label of the $i^{\text{th}}$ inner vertex of even level, which changes from $y$ to $\partial (y)$. These labeled trees are to be regarded as iterated compositions in $\mathcal O \amalg\mathcal F(V)$, since there may be adjacent labels in $\mathcal O$, e.g.
\begin{align*}
&\partial\left(\begin{array}{c}
\begin{tikzpicture}[level distance=7mm, sibling distance=4mm] 
\tikzstyle{every node}=[execute at begin node=$\scriptstyle,execute at end node=$]
\tikzstyle{level 3}=[sibling distance=4mm]
\tikzstyle{level 4}=[sibling distance=2mm]
\node {} [grow'=up]
child{[fill] circle (2pt)
	child{[fill] circle (2pt) 
		child{[fill] circle (2pt) node [above] {x_3}}
		node [left] {y_2}}
	child{}
	child{}
	child{[fill] circle (2pt) node [above] {y_4}}
	node [below right] {x_1}};
\end{tikzpicture}
\end{array}\right)=
(-1)^{\abs{x_1}}\!\!\!
\begin{array}{c}
\begin{tikzpicture}[level distance=7mm, sibling distance=4mm] 
\tikzstyle{every node}=[execute at begin node=$\scriptstyle,execute at end node=$]
\tikzstyle{level 3}=[sibling distance=4mm]
\tikzstyle{level 4}=[sibling distance=2mm]
\node {} [grow'=up]
child{[fill] circle (2pt)
	child{[fill] circle (2pt) 
		child{[fill] circle (2pt) node [above] {x_3}}
		node [left] {\partial(y_2)}}
	child{}
	child{}
	child{[fill] circle (2pt) node [above] {y_4}}
	node [below left] {x_1}};
\end{tikzpicture}
\end{array}
+(-1)^{\abs{x_1}+\abs{y_2}+\abs{x_3}}
\begin{array}{c}
\begin{tikzpicture}[level distance=7mm, sibling distance=4mm] 
\tikzstyle{every node}=[execute at begin node=$\scriptstyle,execute at end node=$]
\tikzstyle{level 3}=[sibling distance=4mm]
\tikzstyle{level 4}=[sibling distance=2mm]
\node {} [grow'=up]
child{[fill] circle (2pt)
	child{[fill] circle (2pt) 
		child{[fill] circle (2pt) node [above] {x_3}}
		node [left] {y_2}}
	child{}
	child{}
	child{[fill] circle (2pt) node [above] {\partial(y_4)}}
	node [below left] {x_1}};
\end{tikzpicture}
\end{array}\\
&\qquad\quad=
(-1)^{\abs{x_1}}\hspace{-27pt}
\begin{array}{c}
\begin{tikzpicture}[level distance=7mm, sibling distance=4mm] 
\tikzstyle{every node}=[execute at begin node=$\scriptstyle,execute at end node=$]
\tikzstyle{level 3}=[sibling distance=4mm]
\tikzstyle{level 4}=[sibling distance=2mm]
\node {} [grow'=up]
child{[fill] circle (2pt)
	child{}
	child{}
	child{[fill] circle (2pt) node [above] {y_4}}
	node [below left] {x_1\circ_1(\partial (y_2)\circ_1 x_3)}};
\end{tikzpicture}
\end{array}
+(-1)^{\abs{x_1}+\abs{y_2}+\abs{x_3}+(\abs{y_4}-1)(\abs{y_2}+\abs{x_3})}\!\!\!
\begin{array}{c}
\begin{tikzpicture}[level distance=7mm, sibling distance=4mm] 
\tikzstyle{every node}=[execute at begin node=$\scriptstyle,execute at end node=$]
\tikzstyle{level 3}=[sibling distance=4mm]
\tikzstyle{level 4}=[sibling distance=2mm]
\node {} [grow'=up]
child{[fill] circle (2pt)
	child{[fill] circle (2pt) 
		child{[fill] circle (2pt) node [above] {x_3}}
		node [left] {y_2}}
	child{}
	child{}
	node [below left] {x_1\circ_4\partial(y_4)}};
\end{tikzpicture}
\end{array}
\end{align*}
\end{rem}

\begin{thm}\label{pasting2}
	Let $\mathcal O\amalg_\partial\mathcal F(V)$ be a twisted coproduct as in Definition \ref{tc}. Assume we have chosen a strong pseudo-cylinder for $\mathcal O$ as in Definition \ref{cilindrodgop}. Then there is a strong pseudo-cylinder for $\mathcal O\amalg_\partial\mathcal F(V)$, that we call \emph{canonical induced strong pseudo-cylinder},
	$$\xymatrix@C=50pt{\mathcal O\amalg_\partial\mathcal F(V)
		\ar@<.5ex>[r]^-{i_0=i_{0}\amalg \mathcal F(i_{0})}
		\ar@<-.5ex>[r]_-{i_1=i_{1}\amalg \mathcal F(i_{1})}
		&
		\mathcal P
		\amalg_{\partial_I}\mathcal F(IV)
		\ar[r]^-{p=p\amalg\mathcal F(p)}
		\ar@(ul,ur)^{h_\partial}
		&
		\mathcal O
		\amalg_\partial\mathcal F(V)
	}$$
	such that 
	the degree $-1$ map $\partial_I\colon V\oplus\Sigma V\oplus V\r\mathcal P$ is
	\[
	\partial_I=\left(\begin{smallmatrix}
	i_0\partial&-h_{\mathcal O}i_1\partial&i_1\partial
	\end{smallmatrix}\right)
	\]
	and
	$$h_\partial=\sum_{n\geq 0}(h_{\mathcal O,V}\partial_I)^nh_{\mathcal O,V},$$
	where $\partial_I$ denotes here the extension to $\mathcal P\amalg\mathcal F(IV)$ in Remark \ref{formulae2}.
\end{thm}

\begin{proof}
	We first check the necessary equation to define the twisted coproduct $\mathcal P
	\amalg_{\partial_I}\mathcal F(IV)$,
	\begin{align*}
		\partial_Id_{IV}+d_{\mathcal P}\partial_I&=
		\left(\begin{smallmatrix}
		i_0\partial&-h_{\mathcal O}i_1\partial&i_1\partial
		\end{smallmatrix}\right)\left(\begin{smallmatrix}
		d_V&1&0\\0&-d_V&0\\0&-1&d_V
		\end{smallmatrix}\right)
		+d_{\mathcal P}
		\left(\begin{smallmatrix}
		i_0\partial&-h_{\mathcal O}i_1\partial&i_1\partial
		\end{smallmatrix}\right)\\
		&=		\left(\begin{smallmatrix}
		i_0\partial d_V&
		i_0\partial+h_{\mathcal O}i_1\partial d_V -i_1\partial
		&i_1\partial d_V
		\end{smallmatrix}\right)
		+		\left(\begin{smallmatrix}
		d_{\mathcal P}i_0\partial&-d_{\mathcal P}h_{\mathcal O}i_1\partial&d_{\mathcal P}i_1\partial
		\end{smallmatrix}\right)\\
				&=		\left(\begin{smallmatrix}
				-i_0d_{\mathcal O}\partial &
				i_0\partial-h_{\mathcal O}i_1d_{\mathcal O}\partial  -i_1\partial
				&-i_1d_{\mathcal O}\partial 
				\end{smallmatrix}\right)
				+		\left(\begin{smallmatrix}
				i_0d_{\mathcal O}\partial&
				-d_{\mathcal P}h_{\mathcal O}i_1\partial
				&i_1d_{\mathcal O}\partial
				\end{smallmatrix}\right)\\
				&=		\left(\begin{smallmatrix}
				0&
				i_0pi_1\partial- i_1\partial - h_{\mathcal O}d_{\mathcal P}i_1\partial   -d_{\mathcal P}h_{\mathcal O}i_1\partial
				&0
				\end{smallmatrix}\right)\\
				&=		\left(\begin{smallmatrix}
				0&
				(i_0p- 1 - h_{\mathcal O}d_{\mathcal P}-d_{\mathcal P}h_{\mathcal O})i_1\partial
				&0
				\end{smallmatrix}\right)\\
				&=		\left(\begin{smallmatrix}
				0&
				0
				&0
				\end{smallmatrix}\right).
	\end{align*}
	Here we use the equations $\partial d_V+d_{\mathcal O}\partial =0$, $i_0p-1_{\mathcal P}=d_{\mathcal P}h_{\mathcal O}+h_{\mathcal O}d_{\mathcal P}$, and $pi_1=1_{\mathcal O}$, and the fact that $i_0,i_1\colon\mathcal O\r\mathcal P$ are DG-maps. The perturbation equation in Lemma \ref{bpl} is a consequence of this. We want to apply this lemma to the sequence of SDRs
	$$\xymatrix@C=60pt{\mathcal O\amalg\mathcal F(V)\ar@<.5ex>[r]^-{i_0}&*+[l]{\mathcal P\amalg\mathcal F(IV)}\ar@<.5ex>[l]^-{p}\ar@(ru,rd)^ {h_{\mathcal O,V}}}$$
	The vanishing condition follows from the criterion in Remark \ref{filtration}. 
	The sequence of subcomplexes $F_n(\mathcal P\amalg\mathcal F(IV))$ is aritywise the direct subsum indexed by the trees $T$ with $\leq n$ vertices of even level. The chain homotopy $h_{\mathcal O,V}$ preserves filtration levels since it restricts to each direct summand. It is easy to see that the extension of $\partial_I$	to $\mathcal P\amalg\mathcal F(IV)$	strictly decreases filtration levels by using the Leibniz rule and the facts that $\partial_I(\mathcal P)=0$ and $\partial_I(IV)\subset\mathcal P$.
	
	The perturbation lemma applies, but we still have to check that the maps $i_0$ and $p$ do not change and that the induced perturbation on $\mathcal O\amalg\mathcal F(V)$ is $\partial$, i.e.
	\begin{align*}
	h_{\mathcal O,V}\Sigma_\infty i_0&=0,&
	p\Sigma_\infty h_{\mathcal O,V}&=0,&
	p\Sigma_\infty i_0&=\partial.
	\end{align*}

	Since $\partial_Ii_0=i_0\partial\colon\mathcal O\amalg\mathcal F(V)\r \mathcal P\amalg\mathcal F(IV)$, $h_{\mathcal O,V}\partial_I i_0=h_{\mathcal O,V}i_0\partial=0\partial=0$. This clearly proves the first equation. Also the third one, since it implies that $p\Sigma_\infty i_0=p\partial_Ii_0=pi_0\partial=\partial$.

	The middle equation follows from $p\partial_Ih_{\mathcal O,V}=0$, that we now check. Since $ph_{\mathcal O}=0\colon\mathcal P\r\mathcal O$ and $ph_I=0\colon IV\r V$, $p$ applied to a labeled tree containing a label of the form $h_{\mathcal O}(x)$, $x\in\mathcal P$, or $h_I(y)$, $y\in IV$, yields $0$. If we apply $\partial_Ih_{\mathcal O,V}$ to a labeled tree in $\mathcal P\amalg\mathcal F(IV)$, we obtain a linear combination of labeled trees, many of them containing labels of the previous form, except for those where the label of an inner vertex of even degree has changed from $y\in IV$ to $\partial_Ih_I(y)$, but
	\[
	p\partial_Ih_I=p\left(\begin{smallmatrix}
	i_0\partial&-h_{\mathcal O}i_1\partial&i_1\partial
	\end{smallmatrix}\right)\left(\begin{smallmatrix}
	0&0&0\\0&0&\sigma\\0&0&0
	\end{smallmatrix}\right)= \left(\begin{smallmatrix}
	0&0&-ph_{\mathcal O}i_1\partial\sigma
	\end{smallmatrix}\right)= \left(\begin{smallmatrix}
	0&0&-0i_1\partial\sigma
	\end{smallmatrix}\right)= \left(\begin{smallmatrix}
	0&0&0
	\end{smallmatrix}\right).
	\]
	Therefore, $p\partial_Ih_{\mathcal O,V}=0$.
\end{proof}

For $\partial=0$, we recover the canonical induced strong pseudo-cylinder of the coproduct. The differential in the canonical strong pseudo-cylinder of an element of the form $\sigma(x)$, $x\in V$, is
\begin{equation}\label{dsigma1}
d\sigma(x)=i_0(x)-i_1(x)-h_{\mathcal O}i_1\partial(x).
\end{equation}

In general, the computation of $h_\partial$ is somewhat involved. Let us just check that it extends the chain homotopies of the canonical strong pseudo-cylinder of $V$ and the chosen strong pseudo-cylinder of $\mathcal O$.

\begin{lem}\label{computa0}
	If $\mathcal P\amalg_{\partial_I}\mathcal F(IV)$ is the canonical induced strong pseudo-cylinder of a twisted coproduct, $x\in\mathcal P$, and $y\in IV$, then 
	\begin{align*}
		h_\partial(x)&=h_{\mathcal O}(x),&
		h_\partial(y)&=h_I(y).
	\end{align*}
\end{lem}

\begin{proof}
	We use Remark \ref{filtration} and the filtration in the proof of Theorem \ref{pasting2}. 
	As an element of $\mathcal P\amalg_{\partial_I}\mathcal F(IV)$, $x\in\mathcal P$ has filtration degree $0$, hence
	\[h_\partial(x)=h_{\mathcal O,V}(x)=h_{\mathcal O}(x).\]
	Moreover, $y\in IV$ has filtration degree $1$, so 
	\begin{align*}
	h_{\partial}(y)&=h_{\mathcal O,V}(y)+h_{\mathcal O}\partial_Ih_{\mathcal O,V}(y)\\
	&=h_I(y)+h_{\mathcal O}\partial_Ih_I(y)\\
	&=h_I(y)+h_{\mathcal O}\left(\begin{smallmatrix}
	i_0\partial&-h_{\mathcal O}i_1\partial&i_1\partial
	\end{smallmatrix}\right)\left(\begin{smallmatrix}
	0&0&0\\0&0&\sigma\\0&0&0
	\end{smallmatrix}\right)(y)\\
	&=h_I(y)+\left(\begin{smallmatrix}
	0&0&-h^2_{\mathcal O}i_1\partial\sigma
	\end{smallmatrix}\right)(y)\\
	&=h_I(y).
	\end{align*}
	Here we use that $h^2_{\mathcal O}=0$.
\end{proof}

We finally consider strong pseudo-cylinders on DG-operads constructed as iterated twisted coproducts.

\begin{defn}\label{rpc}
	A DG-operad $\mathcal O$ is \emph{relatively pseudo-cellular} if it is equipped with an increasing filtration $\{\mathcal O_\beta\}_{\beta\leq\alpha}$ indexed by an ordinal $\alpha$, that we call \emph{length}, which is \emph{exhaustive}, i.e.~$\mathcal O=\mathcal O_\alpha$, \emph{continuous}, i.e.~if $\beta\leq\alpha$ is a limit ordinal then $\mathcal O_\beta=\colim_{\gamma<\beta}\mathcal O_\gamma$ (here $\colim$ can be replaced with $\cup$), and such that if $\beta+1\leq \alpha$ then $\mathcal O_{\beta+1}$ is a twisted coproduct of the following form,
	\[\mathcal O_{\beta+1}=\mathcal O_{\beta}\amalg_{\partial_\beta}\mathcal F(V_\beta).\]
	We say that $\mathcal O$ is \emph{(absolutely) pseudo-cellular} if in addition $\mathcal O_0$ is the initial DG-operad.
\end{defn}

If the complexes of the sequences $V_\beta$ are cofibrant for all $\beta<\alpha$, then the inclusion $\mathcal O_0\r \mathcal O$ is a cofibration. In particular, $\mathcal O$ is cofibrant in the absolute case. If the $V_\beta$ are sequences of free graded modules with trivial differential, then $\mathcal O_0\r \mathcal O$ is a relative cell complex with respect to the standard set of generating cofibrations in the model category of DG-operads, see \cite[Theorem 2.3.11]{hmc} and \cite[proof of Theorem 1.1]{htnso}.

Strong pseudo-cylinders are closed under filtered colimits, for both chain complexes and DG-operads (filtered colimits of DG-operads are computed in the underlying sequences of chain complexes). 

\begin{defn}\label{cspc}
	The \emph{canonical strong pseudo-cylinder} of a relatively pseudo-cellular DG-operad $\mathcal O$
		\[
		\xymatrix{\mathcal O
			\ar@<.5ex>[r]^-{i_0}
			\ar@<-.5ex>[r]_-{i_1}
			&
			I\mathcal O\ar[r]^-{p}\ar@(ul,ur)^{h_{\mathcal O}}
			&
			\mathcal O}
		\]
	is defined by induction on the length $\alpha$ in the following way. If $\alpha=0$, we take the \emph{trivial strong pseudo-cylinder}, 
	$$\xymatrix{\mathcal O 
		\ar@<.5ex>[r]^-{1}
		\ar@<-.5ex>[r]_-{1}
		&
		\mathcal O \ar[r]^-{1}\ar@(ul,ur)^{0}
		&
		\mathcal O.
	}$$
	If $\alpha$ is a limit ordinal, we define the canonical strong pseudo-cylinder of $\mathcal O$ as the colimit of the canonical strong pseudo-cylinders of $\mathcal O_\beta$, $\beta<\alpha$.
	If $\alpha=\beta+1$, the canonical strong pseudo-cylinder of $\mathcal O=\mathcal O_{\beta}\amalg_{\partial_\beta}\mathcal F(V_\beta)$ is defined by applying Theorem \ref{pasting2} to the canonical strong pseudo-cylinder of $\mathcal O_\beta$,
	$$\xymatrix@C=50pt{\mathcal O_\beta\amalg_{\partial_{\beta}}\mathcal F(V_\beta)
		\ar@<.5ex>[r]^-{i_0=i_0\amalg\mathcal F(i_0)}
		\ar@<-.5ex>[r]_-{i_1=i_1\amalg\mathcal F(i_1)}
		&
		I\mathcal O_\beta\amalg_{\partial_{\beta,I}}\mathcal F(IV_\beta)\ar[r]^-{p=p\amalg\mathcal F(p)}\ar@(ul,ur)^{h_{\mathcal O}=h_{\partial_\beta}}
		&
		\mathcal O_{\beta}\amalg_{\partial_{\beta}}\mathcal F(V_\beta),
	}$$
	in particular, $I\mathcal O=I\mathcal O_\beta\amalg_{\partial_{\beta,I}}\mathcal F(IV_\beta)$.
\end{defn}

The DG-operad $I\mathcal O$ is relatively pseudo-cellular of the same length as $\mathcal O$ and $(I\mathcal O)_0=\mathcal O_0$. In particular, if $\mathcal O$ is absolutely pseudo-cellular, then so is $I\mathcal O$. 

If the complexes of the sequences $V_\beta$, $\beta<\alpha$, are cofibrant, the canonical strong pseudo-cylinder of $\mathcal O$ is a relative cylinder in the model theoretic sense, i.e.~$(i_0,i_1)\colon\mathcal O\cup_{\mathcal O_0}\mathcal O\r I\mathcal O$ is a cofibration. In particular, if $\mathcal O$ is absolutely pseudo-cellular, its canonical strong pseudo-cylinder is a strong cylinder.

As a graded operad, any relatively pseudo-cellular DG-operad of length $\alpha$ is $\mathcal O=\mathcal O_0\amalg\mathcal F(V)$, $V=\bigoplus_{\beta<\alpha}V_{\beta}$, and the differential on the free part is determined by the equations $d(x)=d_{V_\beta}(x)+\partial_\beta(x)$, $x\in V_\beta$. The canonical strong pseudo-cylinder, as a graded operad, is $I\mathcal O=\mathcal O_0\amalg\mathcal F(IV)$. On elements of the form $i_0(x)$ and $i_1(x)$, $x\in V_\beta$, the differential is simply given by \[di_j(x)=i_jd(x)=i_jd_{V_\beta}(x)+i_j\partial_\beta(x),\quad j=0,1,\] since $i_0, i_1\colon\mathcal O\r I\mathcal O$ are DG-operad maps. On elements of the form $\sigma x$, $x\in V_\beta$, the differential depends on the involved inductive definition of the chain homotopy 
\begin{equation}\label{dsigma2}
	d\sigma(x)=i_0(x)-i_1(x)-h_{\mathcal O_\beta}i_1\partial_\beta(x),
\end{equation}
see \eqref{dsigma1}. The chain homotopy is easy on $\mathcal O_0$ and on generators of the free part, as a corollary of Lemma \ref{computa0}.

\begin{cor}\label{uno}
	If $\mathcal O$ is a relatively pseudo-cellular DG-operad of length $\alpha$, $x\in\mathcal O_0$, and $y\in IV_\beta$ for some $\beta<\alpha$, then 
	\begin{align*}
	h_{\mathcal O}(x)&=0,&h_{\mathcal O}(y)&=h_I(y).
	\end{align*}
	In particular, $h_{\mathcal O}(\id{})=0$.
\end{cor}

The underlying graded operad of an absolutely pseudo-cellular DG-operad $\mathcal O$ is free, $\mathcal O=\mathcal F(V)$, $V=\bigoplus_{\beta<\alpha}V_{\beta}$, since $\mathcal{O}_0$ is initial and hence it dissapears in coproducts. In order to prove a useful vanishing condition for the chain homotopy, we introduce the following set of linear generators of $I\mathcal O=\mathcal F(V)$ (as a sequence of graded modules).

\begin{defn}
	Let $\mathcal O$ be an absolutely pseudo-cellular DG-operad of length $\alpha$ as above. A \emph{standard labeled tree} in $I\mathcal O$ is a labeled tree such that each label is 
	of the form $i_0(x)$, $i_1(x)$ or $\sigma(x)$, $x\in V_\beta$, $\beta<\alpha$. 
\end{defn}

\begin{lem}\label{vanishing}
	Let $\mathcal O$ be an absolutely pseudo-cellular DG-operad and $t\in I\mathcal O$ a standard labeled tree satisfying one of the two following conditions:	
	\begin{enumerate}
		\item The bottommost label is $\sigma(x)$.
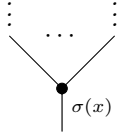
		\[\begin{array}{c}\begin{tikzpicture}[level distance = 7mm, sibling distance = 7mm ]
		\tikzstyle{level 3}=[level distance = 3mm]
		\node {} [grow'=up]
		child{[fill] circle (2pt)
			child{{} node [above] {$\vdots$}}
			child{node {$\cdots$} edge from parent [draw=none]}
			child{{} node [above] {$\vdots$}}
			node [below right] {$\scriptstyle \sigma(x)$}};
		\end{tikzpicture}\end{array}\]
		\item The bottommost label is $i_0(x)$ 
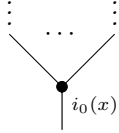
		\[\begin{tikzpicture}[level distance = 7mm, sibling distance = 7mm ]
		\tikzstyle{level 3}=[level distance = 3mm]
		\node {} [grow'=up]
		child{[fill] circle (2pt)
			child{{} node [above] {$\vdots$}}
			child{node {$\cdots$} edge from parent [draw=none]}
			child{{} node [above] {$\vdots$}}
			node [below right] {$\scriptstyle i_0(x)$}};
		\end{tikzpicture}\]
		and $t$ does not contain any \emph{forbidden edge}, i.e.~an inner edge with bottom label $i_0(y)$ and top label $i_1(z)$.
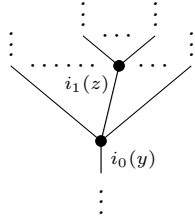
			\[
	\begin{tikzpicture}[level distance = 10mm, sibling distance = 4.8mm ]
	\tikzstyle{level 1}=[level distance = 7mm]
	\tikzstyle{level 2}=[level distance = 10mm]
	\tikzstyle{level 3}=[level distance = 4mm]
	\node [inner sep =0pt] {$\vdots$} [grow'=up]
	child{[fill] circle (2pt)
		child{{} node [above] {$\vdots$}}
		child{node {$\cdots$} edge from parent [draw=none]}
		child{node {$\cdots$} edge from parent [draw=none]}
		child{[fill] circle (2pt)
					child{{} node [above] {$\vdots$}}
					child{node {$\cdots$} edge from parent [draw=none]}
					child{{} node [above] {$\vdots$}}
					node [below left] {$\scriptstyle i_1(z)$}}
		child{node {$\cdots$} edge from parent [draw=none]}
		child{{} node [above] {$\vdots$}}
		node [below right] {$\scriptstyle i_0(y)$}
		edge from parent [solid, thin]};
	\end{tikzpicture}\]
	\end{enumerate}
	Then, $h_{\mathcal O}(t)=0$.
\end{lem}

\begin{proof}
	By induction on the length $\alpha$. If $\alpha=0$, then the chain homotopy is $h_{\mathcal O}=0$, so there is nothing to check. If $\alpha$ is a limit ordinal, then $h_{\mathcal O}(t)=h_{\mathcal O_{\beta}}(t)$ for some $\beta<\alpha$, and $h_{\mathcal O_{\beta}}(t)=0$ by induction hypothesis. 
	
	Suppose $\alpha=\beta+1$. We prove, by induction on the number of inner vertices, that $h_{\mathcal O_\beta,V_\beta}(t)=0$. This clearly suffices.
	
	Suppose $t$ satisfies (1). If $x\in V_\beta$, we can write $t=\sigma(x)(x_1,\dots, x_n)$. The formula in the second paragraph of Remark \ref{formulae} yields
	\begin{align*}
	h_{\mathcal O_\beta,V_\beta}(\sigma(x)(x_1,\dots, x_n))={}&
	h_I(\sigma(x))(x_1,\dots, x_n)\\&\pm\text{terms of the form }i_0p(\sigma(x))(\dots, h_{\mathcal O_\beta,V_\beta}(x_i),\dots).
	\end{align*}
	All summands vanish since $h_I\sigma=0$ and $p\sigma=0$. 
	
	Otherwise, we can write $t=x_0(x_1,\dots,x_n)$. Here, $x_0\in I\mathcal O_\beta$ is a standard labeled tree with the same bottommost label as $t$, $\sigma(x)$. Hence, $h_{\mathcal O_{\beta}}(x_0)=0$ by the first induction hypothesis, and $p(x_0)=0$ since $p\sigma=0$. The formula in the third paragraph of Remark \ref{formulae} yields
	\begin{align}\label{interno}
	h_{\mathcal O_\beta,V_\beta}(x_0(x_1,\dots, x_n))={}&
	h_{\mathcal O_{\beta}}(x_0)(x_1,\dots, x_n)\\
	\nonumber
	&\hspace{-10pt}\pm\text{terms of the form }i_0p(x_0)(\dots, h_{\mathcal O_\beta,V_\beta}(x_i),\dots),
	\end{align}
	which therefore vanishes.
	
	Suppose now that $t$ satisfies (2). If $x\in V_\beta$, we can write the standard labeled tree as $i_0(x)(x_1,\dots, x_n)$, where the elements $x_1,\dots,x_n$ are standard labeled trees with less inner vertices than $t$. We will check that $h_{\mathcal O_\beta,V_\beta}(x_i)=0$, $1\leq i\leq n$. Hence, the formula in the second paragraph of Remark \ref{formulae} yields
	\begin{align*}
	h_{\mathcal O_\beta,V_\beta}(i_0(x)(x_1,\dots, x_n))={}&
	h_Ii_0 (x)(x_1,\dots, x_n)\\&\hspace{-10pt}\pm\text{terms of the form }i_0pi_0(x)(\dots, h_{\mathcal O_\beta,V_\beta}(x_i),\dots),
	\end{align*}
	which vanishes since $h_Ii_0=0$.

	If the standard labeled tree $x_i$ has no inner vertices, then $x_i=\id{}$ and $h_{\mathcal O_\beta,V_\beta}(\id{})=h_0(\id{})=0$. If it has inner vertices, let us look at the bottommost label. It cannot be $i_1(y)$, since it is adjancent to $i_0(x)$ in $t$, and we are assuming that $t$ does not contain forbidden edges. Hence it is $\sigma(y)$ or $i_0(y)$. Not containing a forbidden edge is a property inherited by subtrees. Hence, $x_i$ does not contain forbidden edges. Therefore, $x_i$ satisfies (1) or (2) and has less inner vertices than $t$, so $h_{\mathcal O_\beta,V_\beta}(x_i)=0$ by the second induction hypothesis. 
	
	If $x\in V_\gamma$ for some $\gamma<\beta$, then we can write the standard labeled tree as $x_0(x_1,\dots,x_n)$, where $x_0\in I\mathcal O_\beta$ is a standard labeled tree and, for $1\leq i\leq n$, either $x_i=\id{}$ or $x_i=y_{i,0}(y_{i,1},\dots, y_{i,p_i})$ is a standard labeled tree with less inner vertices than $t$ where $y_{i,0}$ is $i_0(y_i')$, $\sigma(y_i')$ or $i_1(y_i')$ for some $y_i'\in V_\beta$.
	Formula \eqref{interno} also applies in this case. Hence, it sufficies to prove that $h_{\mathcal O_{\beta}}(x_0)=0$ and either $p(x_0)=0$ or $h_{\mathcal O_\beta,V_\beta}(x_i)=0$ for all $1\leq i\leq n$. 
	
	The bottommost label of $x_0$ is the same as in $t$, $i_0(x)$, and $x_0$ cannot contain a forbidded edge, since it is a subtree of $t$. Hence $h_{\mathcal O_{\beta}}(x_0)=0$ by the first induction hypothesis. For $1\leq i\leq n$, if $x_i=\id{}$ then we know that $h_{\mathcal O_\beta,V_\beta}(x_i)=0$. 	Otherwise, let us argue with the possible values of $y_{i,0}$. Being a subtree of $t$, $x_i$ does not contain any forbidden edge. Hence, if $y_{i,0}$ is $i_0(y_i')$ or $\sigma(y_i')$, $h_{\mathcal O_\beta,V_\beta}(x_i)=0$ by the second induction hypothesis, since $x_i$ has less inner vertices than $t$. The bottommost vertex of $x_i$ is adjacent to a vertex in $x_0$. Therefore, if $y_{i,0}=i_1(y_i')$ then the adjacent vertex in $x_0$ is $\sigma(z)$, otherwise $t$ would contain a forbidden edge. In this case $p(x_0)=0$ since $p\sigma=0$.
\end{proof}

%
%
%
%
%

\section{Examples}\label{examples}

We start this section on examples by illustrating how our canonical strong cylinder construction works on the most widely used cellular DG-operad.

\begin{defn}
	The \emph{$A$-infinity DG-operad} $\mathcal A_{\infty}$ is freely generated as a graded operad by
	$$\mu_{n}\in\mathcal A_{\infty}(n)_{n-2},\qquad n\geq 2,$$
	with differential defined by
	\begin{equation*}
	d(\mu_{n})=\sum_{\substack{
			p+q=n+1\\
			1\leq i\leq p}}
	(-1)^{p-i+q(i-1)}\mu_{p}\circ_i\mu_{q}.
	\end{equation*}
\end{defn}

Here we use the sign conventions in \cite{slaic}, but we should point out that Lef\`evre-Hasegawa  uses cohomological grading and, modulo this, an $A$-infinity algebra $(X,m_1,m_2,\dots,m_n,\dots)$ in his sense is an $\mathcal A_\infty$-algebra structure on the chain complex $(X,-m_1)$. This DG-operad is cellular. The following result computes its canonical strong cylinder.

\begin{thm}\label{ainfcyl}
	The canonical strong cylinder of the $A$-infinity DG-operad is the DG-operad $I\mathcal A_\infty$, freely generated as a graded operad by
	\[
	i_0\mu_n,i_1\mu_n\in I\mathcal A_\infty(n)_{n-2},\qquad \sigma\mu_n\in I\mathcal A_\infty(n)_{n-1},\qquad n\geq 2,
	\]
	with differential determined by	the fact that $i_0,i_1\colon\mathcal A_\infty\r I\mathcal A_\infty$ are DG-operad maps and by the following formula, $n\geq 2$,
	\begin{align*}
	d(\sigma\mu_{n})={}&i_0\mu_n-i_1\mu_n-\sum_{\substack{p+q=n+1\\1\leq j\leq p}}(-1)^{p-j+q(j-1)}\sigma\mu_{p}\circ_j i_{1}\mu_{q}\\&+\hspace{-30pt}\sum_{\substack{1\leq s\leq r\\j_0+t_{1}+j_1+\cdots+t_s+j_s=n}}\hspace{-30pt}
	(-1)^{\sum\limits_{k=1}^s(t_k-1)(j_0+\sum\limits_{l=1}^{k-1} (t_l+j_l))}
	i_{0}\mu_{r}(\st{j_0}\dots ,\sigma\mu_{t_{1}},\st{j_1}\dots,\sigma\mu_{t_{2}},\dots, \sigma\mu_{t_{s}},\st{j_s}\dots).
	\end{align*}
\end{thm}

This theorem follows from \eqref{dsigma2} and the last formula in Lemma \ref{tech}. 

\begin{cor}
	Given a chain complex $X$ and two maps to its endomorphism operad $\varphi,\varphi'\colon\mathcal A_\infty\r\mathcal E(X)$, see Remark \ref{opbrace2} below, which correspond with two $A$-infinity structures on $X$, 
	$(X,\{m_n\}_{n\geq 2})$ and $(X,\{m_n'\}_{n\geq 2})$, respectively, a homotopy $H\colon I\mathcal A_\infty\r\mathcal E(X)$ between them, $Hi_0=\varphi'$, $Hi_1=\varphi$, is the same as an $A$-infinity morphism \cite[D\'efinition 1.2.1.2]{slaic} \[\{f_n\}_{n\geq 1}\colon (X,\{m_n\}_{n\geq 2})\To (X,\{m_n'\}_{n\geq 2})\] whose linear part is the identity in $X$, $f_1=1_X$.
\end{cor}

\begin{proof}
	The correspondence is simply given by $m_n'=\varphi'(\mu_n)=H(i_0\mu_n)$, $m_n=\varphi(\mu_n)=H(i_1\mu_n)$, and $f_n=H(\sigma\mu_n)$, $n\geq 2$. 
\end{proof}

In order to simplify computations, we use the formalism of operadic suspensions and brace algebras.

\begin{defn}\label{operadicsusp}
	Given a DG-operad $\mathcal O$, its \emph{operadic suspension} is the DG-operad $\Lambda \mathcal O$ such that $\Lambda \mathcal O(n)=\mathcal O(n)$ as plain modules, $n\geq 0$, with the following new grading
	\[\norm{x}=\abs{x}+1-\text{arity of }x.\]
	The differential is the same as in $\mathcal O$. Compositions in $\Lambda\mathcal O$, that we here denote by $\bullet_i$ in order to avoid confusion, are defined as follows, $x\in\mathcal O(p)$, $y\in\mathcal O(q)$,
	\[x\bullet_iy=(-1)^{\norm{y}(p-i)+\abs{y}(i-1)}x\circ_{i}y,\]
	and the identity is the same $\id{\Lambda \mathcal O}=\id{\mathcal O}$. 
\end{defn}

\begin{rem}\label{opbrace2}
	The functor $\Lambda$ is an an automorphism of the category of DG-operads. It preserves free operads, $\Lambda\mathcal F(V)=\mathcal F(\Lambda V)$, where $\Lambda V$ is defined as above, (twisted) coproducts, (relatively) pseudo-cellular DG-operads, (canonical) strong (pseudo-)cyl\-inders,  etc. It also preserves the homotopical structure (fibrations, cofibrations, and weak equivalences). 
	
	Recall that, given chain complexes $X$ and $Y$, the inner $\hom(X,Y)$ is the chain complex consisting of the modules $\hom(X,Y)_n$ of degree $n$ maps $f\colon X\r Y$ with differential $d(f)=d_Yf-(-1)^{\abs{f}}fd_X$, in particular chain maps $X\r Y$ are $0$-cycles in $\hom(X,Y)$. The endomorphism operad of a chain complex $X$ is \[\mathcal E(X)=\{\hom(X^{\otimes n},X)\}_{n\geq 0},\] 
	the operation $\circ_i$ is composition at the $i^{\text{th}}$ slot, the operadic identity  is the identity map $\id{\mathcal E(X)}=1_X$, and an $\mathcal O$-algebra structure on $X$ is an operad map $\mathcal O\r{\mathcal E(X)}$. There is an isomorphism of DG-operads
	\[\Lambda\mathcal E(X)\cong\mathcal E(\Sigma X),\]
	defined by mapping $f\colon X^{\otimes n}\r X$ to $(-1)^{\abs{f}}\sigma f(\sigma^{-1})^{\otimes n}\colon (\Sigma X)^{\otimes n}\r \Sigma X$, $n\geq 0$. Therefore, an $\mathcal O$-algebra structure on $X$ is the same as a $\Lambda \mathcal O$-algebra structure on $\Sigma X$.

	The \emph{Hadamard product} of two DG-operads $\mathcal O\otimes_{\mathrm H}\mathcal P$ is the DG-operad with $(\mathcal O\otimes_{\mathrm H}\mathcal P)(n)=\mathcal O(n)\otimes\mathcal P(n)$, $n\geq 0$, compositions
	\[(x_1\otimes x_2)\circ_i (y_1\otimes y_2)=(-1)^{\abs{x_2}\abs{y_1}}(x_1\circ_iy_1)\otimes (x_2\circ_iy_2),\]
	and identity $\id{\mathcal O\otimes_{\mathrm H}\mathcal P}=\id{\mathcal O}\otimes\id{\mathcal P}$.
	The operadic suspension $\Lambda \mathcal O$ can be naturally identified with  $\mathcal O\otimes_{\mathrm H}\mathcal E(\Sigma\Bbbk)$. The natural isomorphism
	\[\Lambda \mathcal O\cong \mathcal O\otimes_{\mathrm H}\mathcal E(\Sigma\Bbbk)\]
	maps $x\in\mathcal O(n)$ to $x\otimes(\sigma\varphi_n(\sigma^{-1})^{\otimes n})$, where $\varphi_n\colon \Bbbk^{\otimes n}\r\Bbbk$ is defined by $\varphi_n(1\otimes\cdots\otimes1)=1$.
	
	
\end{rem}

\begin{defn}\label{brace}
	A \emph{graded} or \emph{DG-brace algebra} is a graded module or chain complex $B$ equipped with maps, called \emph{braces}, $n\geq 1$,
	\begin{align*}
	B^{\otimes (n+1)}&\To B,\\
	x_{0}\otimes x_{1}\otimes\cdots\otimes x_{n}&\;\mapsto \; x_{0}\{x_{1},\dots,x_{n}\},
	\end{align*}
	satisfying
	\begin{align*}
	x\{y_{1},\dots,y_{p}\}\{z_{1},\dots z_{q}\}=\hspace{-25pt}\sum_{0\leq i_{1}\leq j_{1}\leq\cdots\leq i_{p}\leq j_{p}\leq q}\hspace{-25pt}(-1)^{\epsilon}x\{
	z_{1},&\dots, z_{i_{1}},
	y_{1}\{z_{i_{1}+1},\dots z_{j_{1}}\},
	\dots\\
	&\dots,
	y_{p}\{z_{i_{p}+1},\dots z_{j_{p}}\},
	z_{j_{p}+1},\dots z_{q}
	\}.
	\end{align*}
		The sign $(-1)^{\epsilon}$ is simply determined by the Koszul sign rule,
		\begin{align*}
		\epsilon&=\sum_{k=1}^{p}\sum_{l=1}^{i_{k}} \abs{y_{k}}\abs{z_{l}}.
		\end{align*}
	
	In the DG-case, the fact that braces are chain maps is equivalent to the following \emph{brace Leibniz rule}:
	\begin{align*}
	d(x_{0}\{x_{1},\dots,x_{n}\})&=d(x_{0})\{x_{1},\dots,x_{n}\}+\sum_{i=1}^{n}(-1)^{\sum\limits_{j=0}^{i-1}\abs{x_{i}}}x_{0}\{x_{1},\dots,d(x_{i}),\dots ,x_{n}\}.
	\end{align*}
\end{defn}

\begin{rem}\label{opbrace}
		If $\mathcal O$ is a graded or DG-operad, then $\bigoplus_{n\geq 0}\mathcal O(n)$ 
		has a brace algebra structure defined as follows. Given $x_{0},\dots,x_n\in\mathcal O$, with $x_0\in\mathcal O(m)$, 
		\begin{align*}
		x_{0}\{x_{1},\dots,x_{n}\}
		&=\sum_{i_{0}+\cdots +i_{n}=m-n}x_{0}(\st{i_{0}}\dots, x_{1},\st{i_{1}}\dots, x_{2},\dots \dots, x_{n},\st{i_{n}}\dots).
		\end{align*}
		Note that $x_{0}\{x_{1},\dots,x_{n}\}=0$ if $n>m$, since the summation is empty in this case. 
\end{rem}

\begin{rem}
	The DG-operad $\Lambda\mathcal A_\infty$ is freely generated as a graded operad by
	$$\mu_{n}\in(\Lambda\mathcal A_{\infty})(n)_{-1},\qquad n\geq 2,$$
	with differential defined by
	\begin{equation*}
	d(\mu_{n})=\sum_{
			p+q=n+1}
	\mu_p\{\mu_q\}.
	\end{equation*}
	This is a cellular DG-operad with $\Lambda \mathcal A_0=\Lambda \mathcal A_1$ the initial DG-operad and, for  $n\geq 2$,
	\[\Lambda\mathcal A_{n}=\Lambda\mathcal A_{n-1}\amalg_{\partial_{n-1}}\mathcal F(\Bbbk\cdot\mu_n),\qquad \partial_{n-1}(\mu_n)=d(\mu_n).\]
	Here $\Bbbk\cdot\mu_n$ denotes the sequence of graded modules freely generated by $\mu_n$ in arity $n$ and degree $-1$ endowed with the trivial differential.
\end{rem}

\begin{rem}\label{braceliketree}
The chain homotopy $h_{V}$ of the canonical strong pseudo-cylinder of a free DG-operad $\mathcal F(V)$ is well behaved with respect to braces. Given $x_0\in IV$  and $x_1,\dots, x_n\in \mathcal F(IV)$,
\begin{align}
\nonumber	h_{V}(x_0\{x_1,\dots,x_n\})={}&h_I(x_0)\{x_1,\dots,x_n\}\\
\nonumber	&\hspace{-40pt}+
	\sum_{i=1}^{n}(-1)^{\sum\limits_{j=0}^{i-1}\abs{x_j}}i_0p(x_0)\{i_0p(x_1),\dots,i_0p(x_{i-1}),h_{V}(x_i),x_{i+1},\dots,x_n\}.
\end{align}
This follows from Remark \ref{hv}.

The chain homotopy $h_{\mathcal O,V}$ of the canonical induced strong pseudo-cylinder of a coproduct $\mathcal O\amalg\mathcal F(V)$, see Definition \ref{cispc}, satisfies the previous formula if $x_0\in IV$ and $x_1,\dots,x_n\in\mathcal P\amalg\mathcal F(IV)$, replacing $h_V$ with $h_{\mathcal O,V}$. Also if $x_0\in\mathcal P$ and each $x_i\in\mathcal P\amalg\mathcal F(IV)$, $1\leq i\leq n$, is either $x_{i}=\id{}$ or $x_i=y_{i,0}\{y_{i,1},\dots,y_{i,p_i}\}$ with $y_{i,0}\in IV$,  replacing $h_V$ with $h_{\mathcal O,V}$ and $h_I$ with $h_{\mathcal O}$. This follows from Remark \ref{formulae}.

%
\end{rem}

The following lemma contains the formulas which prove Theorem \ref{ainfcyl}. The first technical series of formulas is auxiliary. What really matters is the last one.

\begin{lem}\label{tech}
	The following equations hold in $I\Lambda\mathcal A_{n}$, 
		\begin{align}
		\nonumber h_{\Lambda\mathcal A_n}(i_0\mu_r(\dots,\sigma\mu_{t_1},\dots,\sigma\mu_{t_{j-1}},\dots,i_1\mu_{q},\dots,\sigma\mu_{t_{j}},\dots, \sigma\mu_{t_s},\dots))\hspace{-230pt}\\
		\nonumber&\hspace{-30pt}=h_{\Lambda\mathcal  A_{n-1},\Bbbk\cdot\mu_{n}}(\text{the same element})\qquad\text{{for $n=\max\{r,q,t_1,\dots,t_s\}$,}}\\
		\nonumber&=-i_0\mu_r(\dots,\sigma\mu_{t_1},\dots,\sigma\mu_{t_{j-1}},\dots,\sigma\mu_{q},\dots,\sigma\mu_{t_{j}},\dots, \sigma\mu_{t_s},\dots)\\
		&\label{raro}\qquad\quad\parbox[t][1em][t]{0.7\textwidth}{
		if $t_1,\dots,t_{j-1}>r\geq q$,\\
		or if $t_1,\dots,t_{j-1}>q> r$ and $q\leq  t_j,\dots,t_s$;
		}\\[.5cm]
		\nonumber&=0\;\;\;\text{ otherwise;}\\[10pt]
		\nonumber h_{\Lambda\mathcal A_{n}}i_{1}d(\mu_{n+1})&=\sum_{p+q=n+2}\sigma\mu_{p}\{i_{1}\mu_{q}\}-\sum_{\substack{1\leq s\leq r\\t_{1}+\cdots +t_{s}=n+1+s-r}}i_{0}\mu_{r}\{\sigma\mu_{t_{1}},\dots, \sigma\mu_{t_{s}}\}.
		\end{align}

\end{lem}

\begin{proof}
	We simultaneously prove all equations by induction on $n$. Notice that there is nothing to check for $n=0,1$. Take a bigger $n$ and assume that the formulas are true for smaller values. We start with the first series of equations. For the sake of simplicity, let us denote \[x=i_0\mu_r(\dots,\sigma\mu_{t_1},\dots,\sigma\mu_{t_{j-1}},\dots,i_1\mu_{q},\dots,\sigma\mu_{t_{j}},\dots, \sigma\mu_{t_s},\dots).\]
	
	If $n>r,q$, then $n$ is reached at some of the $t_i$s. Hence,
	\begin{align*}
	x&=
	i_0\mu_r(\dots,i_1\mu_{q},\dots)(\dots,\sigma\mu_{n},\dots,\sigma\mu_{n},\dots),
	\end{align*}
	where $i_0\mu_r(\dots,i_1\mu_{q},\dots)$ has been obtained by removing all $\sigma\mu_{n}$s from $x$. Therefore, by Remark \ref{formulae},	
	\begin{align*}
	h_{\Lambda \mathcal A_{n-1},\Bbbk\cdot\mu_{n}}(x)={}&
	h_{\Lambda \mathcal A_{n-1}}(i_0\mu_r(\dots,i_1\mu_{q},\dots))(\dots,\sigma\mu_{n},\dots,\sigma\mu_{n},\dots)\\
	&+\sum i_0p(i_0\mu_r(\dots,i_1\mu_{q},\dots))(\dots,h_I\sigma\mu_n ,\dots).
	\end{align*}
	We do not index the summation since all terms vanish anyway (recall that $h_I\sigma=0$). The maximum subscript in $i_0\mu_r(\dots,i_1\mu_{q},\dots)$ is smaller than $n$. Moreover,  
	$i_0\mu_r(\dots,i_1\mu_{q},\dots)$ satisfies assumption \eqref{raro} if and only if $x$ does. Therefore, using the induction hypothesys, we obtain that
	\begin{align*}
	h_{\Lambda \mathcal A_{n-1},\Bbbk\cdot\mu_n}(x)&=
	-i_0\mu_r(\dots,\sigma\mu_{q},\dots)(\dots,\sigma\mu_n,\dots,\sigma\mu_n,\dots)&\\
	&=-i_0\mu_r(\dots,\sigma\mu_{t_1},\dots,\sigma\mu_{t_{j-1}},\dots,\sigma\mu_{q},\dots,\sigma\mu_{t_{j}},\dots, \sigma\mu_{t_s},\dots)\\
	&\;\;\;\qquad\text{if \eqref{raro} holds,}\\
	&=0\quad\text{otherwise.}
	\end{align*}
	
	If $q=n>r$ and there is some $t_i<n$, then
	\begin{align*}
	x=
	i_0\mu_r(\dots,\sigma\mu_{t_{i}},\dots)(\dots,\sigma\mu_n,\dots,
	i_1\mu_n,\dots,
	\sigma\mu_n,\dots)&,
	\end{align*}
	where $i_0\mu_r(\dots,\sigma\mu_{t_{i}},\dots)$ has been obtained by removing $i_1\mu_n$ and all $\sigma\mu_n$s from $x$. By Remark \ref{formulae},
	\begin{align*}
	h_{\Lambda \mathcal A_{n-1},\Bbbk\cdot\mu_n}(x)={}&
	h_{\Lambda \mathcal A_{n-1}}(i_0\mu_r(\dots,\sigma\mu_{t_{i}},\dots))(\dots,\sigma\mu_n,\dots,
	i_1\mu_n,\dots,
	\sigma\mu_n,\dots)\\
	&-\sum i_0p(i_0\mu_r(\dots,\sigma\mu_{t_{i}},\dots))(\dots)\\={}&0.
	\end{align*}
	The first term is $0$ by Lemma \ref{vanishing} (2). The summation vanishes since $p\sigma=0$.  In this case \eqref{raro} cannot hold.
	
	If $q=t_1=\cdots=t_s=n>r$, then by Remark \ref{formulae},
	\begin{align*}
	h_{\Lambda\mathcal A_{n-1},\Bbbk\cdot\mu_n}(x)={}&
	h_{\Lambda\mathcal A_{n-1}}(i_0\mu_r)(\dots,\sigma\mu_n,\dots,
	i_1\mu_n,\dots,
	\sigma\mu_n,\dots)\\
	&-\sum i_0p(i_0\mu_r)(\dots,h_I\sigma\mu_n,\dots,
	i_1\mu_n,\dots,
	\sigma\mu_n,\dots)\\
	&-	i_0p(i_0\mu_r)(\dots,i_0p\sigma\mu_n,\dots,
	h_Ii_1\mu_n,\dots,
	\sigma\mu_n,\dots)\\
	&+\sum i_0p(i_0\mu_r)(\dots,i_0p\sigma\mu_n,\dots,
	i_0pi_1\mu_n,\dots,
	h_I\sigma\mu_n,\dots).
	\end{align*}
	Using that $h_{\Lambda\mathcal A_{n-1}}i_0=0$, $h_I\sigma=0$, and $p\sigma=0$, we see that all factors vanish unless $i_1\mu_n$ is the first element in the brackets in $x$, i.e.~$j=1$. In that case, there is a single non-vanishing term:
	\begin{align*}
	h_{\Lambda\mathcal A_{n-1},\Bbbk\cdot\mu_n}(x)&=
	-i_0p_0i_0(\mu_r)(\dots,h_Ii_1\mu_n,\dots,\sigma\mu_n,\dots)\\
	&=
	-i_0(\mu_r)(\dots,\sigma\mu_n,\dots,\sigma\mu_n,\dots).
	\end{align*}
	
	If $r=n$, Remark \ref{formulae} also applies directly, and using Corollary \ref{uno},
	\begin{align*}
	h_{\Lambda \mathcal A_{n-1},\Bbbk\cdot\mu_n}(x)={}&
	h_Ii_0\mu_n(\dots,\sigma\mu_{t_1},\dots,\sigma\mu_{t_{j-1}},\dots,i_1\mu_{q},\dots,\sigma\mu_{t_{j}},\dots, \sigma\mu_{t_s},\dots)\\
	&-\sum i_0pi_0(\mu_n)(\dots,h_I\sigma\mu_{t_i},\dots, i_1\mu_{q},\dots)\\
	&-i_0pi_0(\mu_n)(\dots,i_0p\sigma\mu_{t_i},\dots, h_Ii_1\mu_{q},\dots)\\
	&+\sum i_0pi_0(\mu_n)(\dots, i_0p_0i_1\mu_{q},\dots,h_I\sigma\mu_{t_i},\dots).
	\end{align*}
	Using again that $h_Ii_0$, $h_I\sigma=0$, and $p\sigma=0$, we deduce that all terms vanish unless $i_1\mu_{q}$ is the first element in the brackets in $x$, i.e.~$j=1$. In that case, there is again a single non-vanishing term, 
	\begin{align*}
	h_{\Lambda\mathcal A_{n-1},\Bbbk\cdot\mu_n}(x)&=
	-i_0pi_0(\mu_n)(\dots,h_Ii_1\mu_{q},\dots,\sigma\mu_{t_{1}},\dots,\sigma\mu_{t_{s}},\dots)\\
	&=
	-i_0(\mu_n)(\dots,\sigma\mu_{q},\dots,\sigma\mu_{t_{1}},\dots,\sigma\mu_{t_{s}},\dots).
	\end{align*}
	
	We have finally established the formula for $h_{\Lambda\mathcal A_{n-1},\Bbbk\cdot\mu_n}(x)$. Let us see that it concides with $h_{\Lambda\mathcal A_{n}}(x)$. Using the induction hypothesys for the formula of $h_{\Lambda \mathcal A_{n-1}}i_1d(\mu_n)$, we see that $\partial_{n-1,I}h_{\Lambda\mathcal A_{n-1},\Bbbk\cdot\mu_n}(x)$ is a linear combination of standard labeled trees satisfying the hypotheses of Lemma \ref{vanishing} (2). Therefore, 
	\[h_{\Lambda\mathcal A_{n}}\partial_{n-1,I}h_{\Lambda\mathcal A_{n-1},\Bbbk\cdot\mu_n}(x)=0.\] 
	This implies $h_{\Lambda\mathcal A_{n}}(x)=h_{\Lambda\mathcal A_{n-1},\Bbbk\cdot\mu_n}(x)$ by Remark \ref{hrec}.

	We now attack the final formula. In order to apply $h_{\Lambda\mathcal A_{n}}$, we divide $i_{1}d(\mu_{n+1})$ in two three blocks, according to the parity of $n$,
		\begin{align*}
		i_1d(\mu_{n+1})={}&\sum_{\substack{
				p+q=n+2\\p<q}}
			i_{1}\mu_{p}\{i_{1}\mu_{q}\}+\sum_{\substack{
				p+q=n+2\\p>q}}
			i_{1}\mu_{p}\{i_{1}\mu_{q}\}\\
		&+i_{1}\mu_{\frac{n}{2}+1}\{i_{1}\mu_{\frac{n}{2}+1}\}\qquad\text{if $n$ is even},
		\end{align*}
	which will be consiered in this order.

	If $p<q$, $i_{1}\mu_{p}\{i_{1}\mu_{q}\}\in I\Lambda \mathcal A_{q}$ is in filtration degree $1$, hence 
	\begin{align}
\nonumber		h_{\Lambda\mathcal A_{n}}(i_{1}\mu_{p}\{i_{1}\mu_{q}\})={}&h_{\Lambda\mathcal A_{q}}(i_{1}\mu_{p}\{i_{1}\mu_{q}\})\\
\nonumber={}&
		h_{\Lambda\mathcal A_{q-1},\Bbbk\cdot\mu_q}(i_{1}\mu_{p}\{i_{1}\mu_{q}\})\\
		&\label{A} +h_{\Lambda\mathcal A_{q-1}}\partial_{q-1,I}h_{\Lambda\mathcal A_{q-1},\Bbbk\cdot\mu_q}(i_{1}\mu_{p}\{i_{1}\mu_{q}\}).
	\end{align}
	By Remark \ref{braceliketree} and Corollary \ref{uno}, the first summand is
	\begin{align}
		\nonumber h_{\Lambda\mathcal A_{q-1},\Bbbk\cdot\mu_q}(i_{1}\mu_{p}\{i_{1}\mu_{q}\})&=h_{\Lambda\mathcal A_{q-1}}i_1\mu_p\{i_1\mu_q\}
		-i_0pi_1\mu_p\{h_Ii_1\mu_q\}\\
		\label{B}&=\sigma\mu_p\{i_1\mu_q\}-i_0\mu_p\{\sigma\mu_q\}.
	\end{align}
	Let us see that $\eqref{A}=0$. If we apply $\partial_{q-1,I}$ to the summands in \eqref{B}, we obtain
	\begin{align*}
		\partial_{q-1,I}(\sigma\mu_p\{i_1\mu_q\})&=
		\sigma\mu_p\{i_1d(\mu_q)\},&
		\partial_{q-1,I}(i_0\mu_p\{\sigma\mu_q\})&=
			i_0\mu_p\{h_{\Lambda\mathcal A_{q-1}}i_1d(\mu_q)\}.
	\end{align*}
	The first element is a linear combination of standard labeled trees as in Lemma \ref{vanishing} (1). Moreover, using the induction hypothesis for $h_{\Lambda\mathcal A_{q-1}}i_1d(\mu_q)$ we see that the second term is a linear combination of standard labeled trees as in Lemma \ref{vanishing} (2). Therefore, this lemma proves that
	\begin{align*}
		h_{\Lambda\mathcal A_{q-1}}(\sigma\mu_p\{i_1d(\mu_q)\})&=0,&
				h_{\Lambda\mathcal A_{q-1}}(i_0\mu_p\{h_{\Lambda\mathcal A_{q-1}}i_1d(\mu_q)\})&=0.
	\end{align*}

	If $p>q$, $i_{1}\mu_{p}\{i_{1}\mu_{q}\}\in I\Lambda \mathcal A_{p}$ is in filtration degree $1$ and
	\begin{align}
\nonumber	h_{\Lambda\mathcal A_{n}}(i_{1}\mu_{p}\{i_{1}\mu_{q}\})={}&h_{\Lambda\mathcal A_{p}}(i_{1}\mu_{p}\{i_{1}\mu_{q}\})\\\nonumber={}&
	h_{\Lambda\mathcal A_{p-1},\Bbbk\cdot\mu_p}(i_{1}\mu_{p}\{i_{1}\mu_{q}\})\\
\label{BB'}	&+h_{\Lambda\mathcal A_{p-1}}\partial_{p-1,I}h_{\Lambda\mathcal A_{p-1},\Bbbk\cdot\mu_p}(i_{1}\mu_{p}\{i_{1}\mu_{q}\}).
	\end{align}
	As above,
	\begin{align}
\nonumber		h_{\mathcal A_{p-1},\Bbbk\cdot\mu_p}(i_{1}\mu_{p}\{i_{1}\mu_{q}\})&=
		h_Ii_1\mu_p\{i_1\mu_q\}-i_0pi_1\mu_p\{h_{\Lambda\mathcal A_{p-1}}i_1\mu_q\}\\
\label{AA'}		&=\sigma\mu_p\{i_1\mu_q\}-i_0\mu_p\{\sigma\mu_q\}.
	\end{align}
	In this case, \eqref{BB'}  need not be $0$. Using the induction hypothesis for $h_{\Lambda\mathcal A_{p-1}}i_1d(\mu_p)$, we see that $\partial_{p-1,I}$ of the two summands in \eqref{AA'} is
	\begin{align*}
		\partial_{p-1,I}(i_0\mu_p\{\sigma\mu_q\})={}&i_0d(\mu_p)\{\sigma\mu_q\},\\
		\partial_{p-1,I}(\sigma\mu_p\{i_1\mu_q\}) ={}&-h_{\Lambda\mathcal A_{p-1}}i_1d(\mu_p)\{i_1\mu_q\}\\
		={}&-\sum_{k+l=p+1}\sigma\mu_k\{i_1\mu_l\}\{i_1\mu_q\}\\
		&+\sum_{\substack{1\leq s\leq r\\t_1+\cdots+t_s=p+s-r}}i_0\mu_r\{\sigma\mu_{t_1},\dots,\sigma\mu_{t_s}\}\{i_1\mu_q\}.
	\end{align*}
	The brace relation yields
	\begin{align*}
		\sigma\mu_k\{i_1\mu_l\}\{i_1\mu_q\}={}&\sigma\mu_k\{i_1\mu_l,i_1\mu_q\}+\sigma\mu_k\{i_1\mu_l\{i_1\mu_q\}\}\\&-\sigma\mu_k\{i_1\mu_q,i_1\mu_l\},\\
		i_0\mu_r\{\sigma\mu_{t_1},\dots,\sigma\mu_{t_s}\}\{i_1\mu_q\}={}&
		\sum_{j=1}^si_0\mu_r\{\dots,\sigma\mu_{t_{j}}\{i_1\mu_q\},\dots\}\\&
		+\sum_{j=1}^{s+1}i_0\mu_r\{\dots,\sigma\mu_{t_{j-1}},i_1\mu_q,\sigma\mu_{t_j},\dots\}\text{ if }s<r,
	\end{align*}
	see Remark \ref{opbrace}.
	
	By Lemma \ref{vanishing},
	\begin{align*}
		h_{\Lambda \mathcal A_{p-1}}(\sigma\mu_k\{i_1\mu_l\}\{i_1\mu_q\})&=0,&
		h_{\Lambda \mathcal A_{p-1}}(i_0\mu_r\{\dots,\sigma\mu_{t_{j}}\{i_1\mu_q\},\dots\})&=0.
	\end{align*}
	Moreover, by the first series of equations in the statement, already checked up to $n$,
		\begin{align*}
		h_{\Lambda \mathcal A_{p-1}}(i_0\mu_r\{\dots,\sigma\mu_{t_{j-1}},i_1\mu_q,\sigma\mu_{t_j},\dots\})&=
		\left\{
		\begin{array}{ll}
		-i_0\mu_r\{\dots,\sigma\mu_{t_{j-1}},\sigma\mu_q,\sigma\mu_{t_j},\dots\}\\
		\qquad\text{if }t_1,\dots,t_{j-1}>r\geq q,\\
		\qquad\text{or if }t_1,\dots,t_{j-1}>q> r\\\qquad\quad\!\text{ and }q\leq  t_j,\dots,t_s;\\[10pt]
		0,\quad\,\text{otherwise.}
		\end{array}
		\right.
		\end{align*}
	
	If $n$ is even, by Remark \ref{hrec},
	\begin{align*}
	h_{\Lambda\mathcal A_n}(i_1\mu_{\frac{n}{2}+1}\{i_1\mu_{\frac{n}{2}+1}\})={}&h_{\Lambda\mathcal A_{\frac{n}{2}},\Bbbk\cdot\mu_{\frac{n}{2}+1}}(i_1\mu_{\frac{n}{2}+1}\{i_1\mu_{\frac{n}{2}+1}\})\\
	&+h_{\Lambda\mathcal A_{\frac{n}{2}
			+1}}\partial_{\frac{n}{2},I}h_{\Lambda\mathcal A_{\frac{n}{2}},\Bbbk\cdot\mu_{\frac{n}{2}+1}}(i_1\mu_{\frac{n}{2}+1}\{i_1\mu_{\frac{n}{2}+1}\}).\\
	\end{align*}
	The first summand is computed as in the previous two cases,
	\begin{align*}
	h_{\Lambda\mathcal A_{\frac{n}{2}},\Bbbk\cdot\mu_{\frac{n}{2}+1}}(i_1\mu_{\frac{n}{2}+1}\{i_1\mu_{\frac{n}{2}+1}\})&=h_Ii_1\mu_{\frac{n}{2}+1}\{i_1\mu_{\frac{n}{2}+1}\}-i_0pi_1\mu_{\frac{n}{2}+1}\{h_Ii_1\mu_{\frac{n}{2}+1}\}\\
	&=\sigma\mu_{\frac{n}{2}+1}\{i_1\mu_{\frac{n}{2}+1}\}-i_0\mu_{\frac{n}{2}+1}\{\sigma\mu_{\frac{n}{2}+1}\},
	\end{align*}
	Let us check that the second one vanishes. We have
	\begin{align*}
		\partial_{\frac{n}{2},I}(i_0\mu_{\frac{n}{2}+1}\{\sigma\mu_{\frac{n}{2}+1}\})&=i_0d(\mu_{\frac{n}{2}+1})\{\sigma\mu_{\frac{n}{2}+1}\}+i_0\mu_{\frac{n}{2}+1}\{h_{\Lambda\mathcal A_{\frac{n}{2}}}i_1d(\mu_{\frac{n}{2}+1})\},\\	\partial_{\frac{n}{2},I}(\sigma\mu_{\frac{n}{2}+1}\{i_1\mu_{\frac{n}{2}+1}\})&=-h_{\Lambda\mathcal A_{\frac{n}{2}}}i_1d(\mu_{\frac{n}{2}+1})\{i_1\mu_{\frac{n}{2}+1}\}+\sigma\mu_{\frac{n}{2}+1}\{i_1d(\mu_{\frac{n}{2}+1})\}.
	\end{align*}
	Clearly $i_0d(\mu_{\frac{n}{2}+1})\{\sigma\mu_{\frac{n}{2}+1}\}$ and $\sigma\mu_{\frac{n}{2}+1}\{i_1d(\mu_{\frac{n}{2}+1})\}$ are linear combinations of standard labeled trees as in Lemma \ref{vanishing}. Using the formula for $h_{\Lambda\mathcal A_{\frac{n}{2}}}i_1d(\mu_{\frac{n}{2}+1})$, that we know by induction hypothesis, we see that $i_0\mu_{\frac{n}{2}+1}\{h_{\Lambda\mathcal A_{\frac{n}{2}}}i_1d(\mu_{\frac{n}{2}+1})\}$ too. Therefore, 
		\begin{align*}
		h_{\Lambda\mathcal A_{\frac{n}{2}
				+1}}(i_0d(\mu_{\frac{n}{2}+1})\{\sigma\mu_{\frac{n}{2}+1}\})&=0,&
	h_{\Lambda\mathcal A_{\frac{n}{2}
			+1}}(i_0\mu_{\frac{n}{2}+1}\{hi_1d(\mu_m)\})&=0,\\&&
		h_{\Lambda\mathcal A_{\frac{n}{2}
				+1}}(\sigma\mu_{\frac{n}{2}+1}\{i_1d(\mu_{\frac{n}{2}+1})\})&=0.
	\end{align*}
	Moreover,
	\begin{align}
		\nonumber h_{\Lambda\mathcal A_{\frac{n}{2}}}i_1d(\mu_{\frac{n}{2}+1})\{i_1\mu_{\frac{n}{2}+1}\}={}&
		\sum_{p+q=\frac{n}{2}+2}\sigma\mu_{p}\{i_{1}\mu_{q}\}\{i_1\mu_{\frac{n}{2}+1}\}\\\label{B'}&
		-\sum_{\substack{1\leq s\leq r\\t_{1}+\cdots +t_{s}=\frac{n}{2}+1+s-r}}i_{0}\mu_{r}\{\sigma\mu_{t_{1}},\dots, \sigma\mu_{t_{s}}\}\{i_1\mu_{\frac{n}{2}+1}\}.
	\end{align}
	Using the brace equation, we can check as above that $\sigma\mu_{p}\{i_{1}\mu_{q}\}\{i_1\mu_{\frac{n}{2}+1}\}$  is  a linear combination of standard labeled trees as in Lemma \ref{vanishing} (1), so
	\begin{align*}
		h_{\Lambda\mathcal A_{\frac{n}{2}
				+1}}(\sigma\mu_{p}\{i_{1}\mu_{q}\}\{i_1\mu_{\frac{n}{2}+1}\})&=0.
	\end{align*}
	Furthermore, the summands in \eqref{B'} are
	\begin{align}
		\nonumber i_0\mu_{r}\{\sigma\mu_{t_{1}},\dots, \sigma\mu_{t_{s}}\}\{i_1\mu_{\frac{n}{2}+1}\}={}&
		\sum_{j=1}^si_0\mu_r\{\dots,\sigma\mu_{t_{j}}\{i_1\mu_{\frac{n}{2}+1}\},\dots\}\\
		\label{C}&
		+\sum_{j=1}^{s+1}i_0\mu_r\{\dots,\sigma\mu_{t_{j-1}},i_1\mu_{\frac{n}{2}+1},\sigma\mu_{t_j},\dots\}\text{ if }s<r,
	\end{align}
	see Remark \ref{opbrace}. Here, $i_0\mu_r\{\dots,\sigma\mu_{t_{j}}\{i_1\mu_{\frac{n}{2}+1}\},\dots\}$ is a linear combination of standard labeled trees as in Lemma \ref{vanishing} (2), so
	\begin{align*}
		h_{\Lambda\mathcal A_{\frac{n}{2}
				+1}}(i_0\mu_r\{\dots,\sigma\mu_{t_{j}}\{i_1\mu_{\frac{n}{2}+1}\},\dots\})&=0.
	\end{align*}
	Looking at the index of the summation in \eqref{B'}, we see that, in \eqref{C}, $r,t_1,\dots,t_s\leq \frac{n}{2}$. Therefore, by the first series of equations in the statement, already checked up to $n$, we see that 
	\begin{align*}
		h_{\Lambda\mathcal A_{\frac{n}{2}
				+1}}(i_0\mu_r\{\dots,\sigma\mu_{t_{j-1}},i_1\mu_{\frac{n}{2}+1},\sigma\mu_{t_j},\dots\})={}&0.
	\end{align*}

	Collecting previous equations, we obtain
	\begin{align*}
		hi_1d(\mu_{n+1})={}&
		\sum_{p+q=n+2}(\sigma\mu_p\{i_1\mu_q\}-i_0\mu_p\{\sigma\mu_q\})\\
		&-\overbrace{\sum_{\substack{p+q=n+2\\p>q}}\sum_{s=1}^{r-1}\sum_{j=1}^{s+1}\sum_{\substack{t_1+\cdots+t_s=p+s-r\\\text{and either}\\t_1,\dots,t_{j-1}>r\geq q\\\text{or}\\t_1,\dots,t_{j-1}>q>r\\
				\text{and }t_j,\dots,t_s\geq q}}i_0\mu_r\{\dots,\sigma\mu_{t_{j-1}},\sigma\mu_q,\sigma\mu_{t_j},\dots\}}^{(\star)}.
	\end{align*}
	Now, it is enough to check that
	\begin{align}\label{D}
		(\star)
		&=\sum_{\substack{2\leq \bar{s}\leq \bar{r}\\\bar t_{1}+\cdots +\bar t_{\bar s}=n+1+\bar s-\bar r}}i_{0}\mu_{\bar r}\{\sigma\mu_{\bar t_{1}},\dots, \sigma\mu_{\bar t_{\bar s}}\}.
	\end{align}
	The summand in $(\star)$ corresponding to certain $p,q,j,r,s,t_1,\dots,t_s$ is the same as the summand on the right hand side of \eqref{D} corresponding to 
	\begin{align*}
		\bar r&=r,&
		\bar s&=s+1,&
		\bar t_i&=t_i\text{ for }1\leq i<j,&
		\bar t_j&=q,&
		\bar t_i&=t_{i-1}\text{ for }j<i\leq \bar s.
	\end{align*}
	Obviously $\bar s=s+1\geq 1+1=2$; $\bar s=s+1\leq r-1+1=\bar r$; and, using that $p+q=n+2$,
	\[\bar t_1+\cdots+\bar t_{\bar s}=t_1+\cdots+t_s+q=p+s-r+q=n+1+s+1-r=n+1+\bar s-\bar r.\]
	
	Consider now a summand on the right hand side of \eqref{D}, corresponding to certain $\bar r,\bar s,\bar t_1,\dots,\bar t_{\bar s}$. Suppose that some $\bar t_i$ is smaller or equal than $\bar r$. Let $j$ be the smallest value $1\leq j\leq \bar s$ such that $\bar t_{j}\leq\bar r$, in particular $\bar t_1,\dots,\bar t_{j-1}>\bar r\geq \bar t_{j}$. Then the summand on the right hand side of \eqref{D} is the same as the summand in $(\star)$ corresponding to
	\begin{align}
		\label{gurdu2} r&=\bar r,&t_i&=\bar t_i\text{ for }1\leq i<j,\\
		\nonumber s&=\bar s-1,&t_{i-1}&=\bar t_i\text{ for }j<i\leq\bar s,\\
		\nonumber q&=\bar t_{j},&p&=\bar t_1+\cdots+\bar t_{j-1}+\bar t_{j+1}+\cdots+\bar t_{\bar s}+\bar r-\bar s+1.
	\end{align}
	Note that $s=\bar s-1\geq 2-1=1$; $1\leq j\leq\bar s=s+1$;  $r=\bar r\geq \bar s>\bar s-1=s$; 
	$p+q=\bar t_1+\cdots+\bar t_{\bar s}+\bar r-\bar s+1=n+2$; for $1\leq i<j$, $t_i=\bar t_i>\bar r=r\geq \bar t_{j}=q$; and, since $\bar t_i> 1$ for all $i$, $p> \bar s-1+\bar r-\bar s+1=\bar r\geq \bar t_j=q$.
	
	Othwerwise, all $\bar t_i$s are bigger than $\bar r$. Let $j$ be the smallest $1\leq j\leq \bar s$ such that $\bar t_j$ attains the minimum value among all $\bar t_i$s, in particular $\bar t_1,\dots,\bar t_{j-1}>\bar t_j>\bar r$ and $\bar t_{j+1},\dots,\bar t_{\bar s}\geq \bar t_j$. Once again, one can straightforwardly check that the summand on the right hand side of \eqref{D} is the same as the summand in $(\star)$ corresponding to the formulas in \eqref{gurdu2}. 
\end{proof}

We now compute examples with non-trivial operations in arities $1$ and $0$.

\begin{exm}\label{conder}
	We can consider the following extension $\mathcal{A}^D_{\infty}$ of the $A$-infinity operad which has non-trivial elements in arity $1$. As a graded operad, it is freely generated by
	$$\mu_{n}\in\mathcal{A}^{D}_{\infty}(n)_{n-2},\quad n\geq 2;
	\qquad D_{n}\in\mathcal{A}^{D}_{\infty}(n)_{n-1},\quad n\geq 1.$$
	The differential on the new generators is given by
	\begin{align*}
	d(D_{n})={}
	&\sum_{\substack{
			p+q=n+1\\
			1\leq i\leq p}}
	((-1)^{(q-1)(i-1)}\mu_{p}\circ_iD_{q}-(-1)^{p-i+q(i-1)}D_{p}\circ_i\mu_{q}).
	\end{align*}
	%
	This is the Koszul resolution of the operad whose algebras are associative algebras equipped with a degree $0$ derivation, considered by Loday in \cite{oaad}. We have simplified notation and adapted sign conventions to our setting. We should warn the reader that \cite{oaad} contains an obvious mistake in the grading of generators coming from the $A$-infinity operad (in fact, Loday's differential would not be homogeneous with his grading). We have fixed this mistake and its consequences on signs. Let us sketch how the canonical strong cylinder $I\mathcal{A}^D_{\infty}$ can be computed as above, by using operadic suspensions.
			
	The operadic suspension $\Lambda\mathcal A^D_\infty$ is freely generated as a graded operad by
	$$\mu_{n}\in\Lambda\mathcal{A}^{D}_{\infty}(n)_{-1},\quad n\geq 2;
	\qquad D_{n}\in\Lambda\mathcal{A}^{D}_{\infty}(n)_{0},\quad n\geq 1.$$
	The differential of the generators not coming from $\Lambda\mathcal A_\infty$ is then
	\begin{align*}
		d(D_{n})={}&\sum_{p+q=n+1}(\mu_{p}\{D_{q}\}-D_{p}\{\mu_{q}\}).
	\end{align*}
	This operad is cellular with $\Lambda\mathcal A_0^D=\Lambda\mathcal A_1^D$ the initial DG-operad and 
	\begin{align*}
		\Lambda \mathcal A_n^D&=\Lambda \mathcal A_{n-1}^D\amalg_{\partial_{n-1}}\mathcal F(\Bbbk\cdot
		\{D_{n-1},\mu_n\}),\quad n\geq 2.
	\end{align*}
	Here $\partial_{n-1}$ is defined as the differential. Lemma \ref{tech} can be extended to show
	\begin{align*}
	  h_{\Lambda\mathcal A_{n}^{D}}i_{1}d(D_{n})&=\sum_{p+q=n+1}\sigma\mu_{p}\{i_{1}D_{q}\}\\
	 &\quad-\hspace{-20pt}\sum_{\substack{0\leq s< r\\t_{1}+\cdots +t_{s}+q=n+1+s-r\\
			1\leq j\leq s+1
		}}\hspace{-20pt}i_{0}\mu_{r}\{\sigma\mu_{t_{1}},\dots, 
		\sigma\mu_{t_{j-1}},\sigma D_{q},\sigma\mu_{t_{j}},\dots,
		\sigma\mu_{t_{s}}\}
		\\
		 &\quad
		-\sum_{p+q=n+1}\sigma D_{p}\{i_{1}\mu_{q}\}
		-\hspace{-20pt}\sum_{\substack{1\leq s\leq r\\t_{1}+\cdots +t_{s}=n+s-r}}\hspace{-20pt}i_{0}D_{r}\{\sigma\mu_{t_{1}},\dots, \sigma\mu_{t_{s}}\}.
	\end{align*}
	This formula, together with \eqref{dsigma2} and the fact that the inclusions $I\Lambda \mathcal A_\infty\subset I\Lambda \mathcal A_\infty^D$ and $i_0,i_1\colon \Lambda \mathcal A_\infty^D\r I\Lambda \mathcal A_\infty^D$ are DG-operad morphisms, completely determines $I\Lambda \mathcal A_\infty^D$ as a DG-operad, and hence $I\mathcal A_\infty^D$.
\end{exm}

	Our canonical strong pseudo-cylinder generalizes the classical cylinder of 
	DG-algebras, regarded as DG-operads concentrated in arity $1$. This is a consequence the following result, see \eqref{dsigma2} and \cite[\S I.7]{ah}.
	
\begin{lem}
Let $\mathcal O$ be an absolutely pseudo-cellular DG-operad of length $\alpha$ as in Definition \ref{rpc}, such that $V_\beta$ is concentrated in arities $0$ and $1$, $\beta<\alpha$. Then, the following equation holds for any $x\in\mathcal O(1)$ and any $y\in\mathcal O$,
	\begin{align*}
		h_{\mathcal O}i_1(x\circ_1y)&=h_{\mathcal O}i_1(x)\circ_1i_1(y)+(-1)^{\abs{x}}i_0(x)\circ_1 h_{\mathcal O}i_1(y).
	\end{align*}
\end{lem}

\begin{proof}
	A simple computation shows that the statement follows if we prove that, given $x_i\in V_{\beta_i}(1)$, $1\leq i< n$, and $x_n\in V_{\beta_n}$,
	\begin{align*}
		\hspace{10pt}&\hspace{-10pt}h_{\mathcal O}i_1(x_1\circ_1\cdots\circ_1x_n)\\
		&=h_{\mathcal O_{\gamma+1}}i_1(x_1\circ_1\cdots\circ_1x_n)\\
		&=h_{\mathcal O_\gamma,V_\gamma}i_1(x_1\circ_1\cdots\circ_1x_n)
		\\&=\sum_{j=1}^n(-1)^{\sum\limits_{k={1}}^{j-1}\abs{x_k}} i_0(x_1)\circ_1\cdots\circ_1i_0(x_{j-1})\circ_1\sigma(x_j)\circ_1i_1(x_{j+1})
		\circ_1\cdots\circ_1i_1(x_n),
	\end{align*}
	where $\gamma=\max\{\beta_1,\dots,\beta_n\}$. The first equation holds by definition. We check the other two ones by induction, first on $\gamma$ and then on the filtration degree of $x_1\circ_1\cdots\circ_1x_n\in\mathcal O_\gamma\amalg_{\partial_\gamma}\mathcal F(V_\gamma)$ with respect to the filtration in the proof of Theorem \ref{pasting2}, i.e.~on the (positive) amount of numbers $1\leq i\leq n$ with $\beta_i=\gamma$. Let $i$ be the smallest $1\leq i\leq n$ such that $\beta_i=\gamma$. Then 
	\begin{align*}
	\hspace{20pt}&\hspace{-20pt}h_{\mathcal O_\gamma,V_\gamma}(x_1\circ_1\cdots\circ_1x_n)\\
	={}&h_{\mathcal O_\gamma}i_1(x_1\circ_1\cdots\circ_1x_{i-1})\circ_1i_1(x_i\circ_1\cdots\circ_1x_n)\\
	&+(-1)^{\sum\limits_{j=0}^{i-1}\abs{x_j}}i_0pi_1(x_1\circ_1\cdots\circ_1x_{i-1})\circ_1h_{\mathcal O_\gamma,V_\gamma}i_1(x_i\circ_1\cdots\circ_1x_n)\\
	={}&h_{\mathcal O_\gamma}i_1(x_1\circ_1\cdots\circ_1x_{i-1})\circ_1i_1(x_i\circ_1\cdots\circ_1x_n)\\
	&+(-1)^{\sum\limits_{j=1}^{i-1}\abs{x_j}}i_0pi_1(x_1\circ_1\cdots\circ_1x_{i-1})\circ_1h_{I}i_1(x_i)\circ_1i_1(x_{i+1}\circ_1\cdots\circ_1x_n)\\
	&+(-1)^{\sum\limits_{j=1}^i\abs{x_j}}i_0pi_1(x_1\circ_1\cdots\circ_1x_{i-1})\circ_1i_0pi_1(x_i)\circ_1h_{\mathcal O_\gamma,V_\gamma}i_1(x_{i+1}\circ_1\cdots\circ_1x_n)\\
	={}&\sum_{j=1}^{i-1}(-1)^{\sum\limits_{k=1}^{j-1}\abs{x_k}}i_0(x_1)\circ_1\cdots\circ_1i_0(x_{j-1})\circ_1\sigma(x_j)\circ_1i_1(x_{j+1})\circ_1\cdots\circ_1i_1(x_n)\\
	&+(-1)^{\sum\limits_{j=1}^{i-1}\abs{x_j}}i_0(x_1)\circ_1\cdots\circ_1i_0(x_{i-1})\circ_1\sigma(x_i)\circ_1i_1(x_{i+1})\circ_1\cdots i_1(\circ_1x_n)\\
	&+\sum_{j=i+1}^{n}(-1)^{\sum\limits_{k=1}^{j-1}\abs{x_k}}i_0(x_1)\circ_1\cdots\circ_1i_0(x_{j-1})\circ_1\sigma(x_j)\circ_1i_1(x_{j+1})\circ_1\cdots\circ_1i_1(x_n)\\
	={}&\sum_{j=1}^n(-1)^{\sum\limits_{k={1}}^{j-1}\abs{x_k}} i_0(x_1)\circ_1\cdots\circ_1i_0(x_{j-1})\circ_1\sigma(x_j)\circ_1i_1(x_{j+1})
	\circ_1\cdots\circ_1i_1(x_n).
	\end{align*}
	In the first two equations, we use Remark \ref{formulae}. In the third one, we apply the two induction hypotheses. Indeed, $x_1\circ_1\cdots\circ_1x_{i-1}\in\mathcal O_\gamma$ and either $x_{i+1}\circ_1\cdots\circ_1x_n\in\mathcal O_\gamma\amalg\mathcal F(V_\gamma)$ has positive but smaller filtration degree or $x_{i+1}\circ_1\cdots\circ_1x_n\in\mathcal O_\gamma$. 
	
	In order to check that $h_{\mathcal O_{\gamma+1}}(x_1\circ_1\cdots\circ_1x_n)=h_{\mathcal O_\gamma,V_\gamma}(x_1\circ_1\cdots\circ_1x_n)$, observe that
	\begin{align*}
		\hspace{-15pt}&\hspace{-15pt}\partial_{\gamma,I}(i_0(x_1)\circ_1\cdots\circ_1i_0(x_{j-1})\circ_1\sigma(x_j)\circ_1i_1(x_{j+1})
		\circ_1\cdots\circ_1i_1(x_n))\\
		={}&-(-1)^{\sum\limits_{k=1}^{j-1}\abs{x_k}}i_0(x_1)\circ_1\cdots\circ_1i_0(x_{j-1})\circ_1h_{\mathcal O_\gamma}i_1\partial_{\gamma}(x_j)\circ_1i_1(x_{j+1})
		\circ_1\cdots\circ_1i_1(x_n)\\
		&-(-1)^{\sum\limits_{k=1}^j\abs{x_k}}i_0(x_1)\circ_1\cdots\circ_1i_0(x_{j-1})\circ_1\sigma(x_j)\circ_1i_1\partial_{\gamma}(x_{j+1}
		\circ_1\cdots\circ_1x_n).
	\end{align*}
	The second summand is clearly a linear combination of standard labeled trees as in Lemma \ref{vanishing}, hence
	\[h_{\mathcal O_{\gamma+1}}(i_0(x_1)\circ_1\cdots\circ_1i_0(x_{j-1})\circ_1\sigma(x_j)\circ_1i_1\partial_{\gamma}(x_{j+1}
	\circ_1\cdots\circ_1x_n))=0.\]
	Moreover, using the induction hypothesis on $h_{\mathcal O_\gamma}i_1$, we see that the first summand is also such a linear combination, therefore
	\[h_{\mathcal O_\gamma,V_\gamma}(i_0(x_1)\circ_1\cdots\circ_1i_0(x_{j-1})\circ_1h_{\mathcal O_\gamma}i_1\partial_{\gamma}(x_j)\circ_1i_1(x_{j+1})
	\circ_1\cdots\circ_1i_1(x_n))=0.\]
	This proves that
	$h_{\mathcal O_{\gamma+1}}(x_1\circ_1\cdots\circ_1x_n)=h_{\mathcal O_\gamma,V_\gamma}(x_1\circ_1\cdots\circ_1x_n)$ by Remark \ref{hrec}.
\end{proof}

In the conditions of the previous lemma, the operad $\mathcal O$ consists of just a DG-algebra $\mathcal O(1)$ and a left $\mathcal O(1)$-module $\mathcal O(0)$. It is trivial in  higher arities.	The full computation of $h_{\mathcal O}i_{1}$ has been possible in this case since all labeled trees in $\mathcal O$ are linear and  the path order coincides with the linear order.

\section{Linear DG-operads}\label{linear}

In this section we analyze the canonical strong pseudo-cylinder construction in a class of relatively pseudo-cellular DG-operads, that we call \emph{linear}, where formulas are easy. Classical examples, such as the $A$-infinity DG-operad, are not linear, but relative examples do show up, as we will see below. 

Given a graded operad $\mathcal O$, recall that an \emph{$\mathcal O$-module} \cite[Definition 1.4]{mfo} is a sequence of graded modules $M=\{M(n)\}_{n\geq 0}$ equipped with compositions, $1\leq i\leq p$, $q\geq 0$
\[\circ_i\colon M(p)\otimes \mathcal O(q)\To M(p+q-1),\qquad \circ_i\colon\mathcal O(p)\otimes M(q)\To M(p+q-1),
\]
satisfying the same laws as graded operads \eqref{operadlaw} when one of the variables is in $M$ and the rest in $\mathcal O$. These are the same as the linear modules introduced in \cite[Definition 2.13]{cmmc} and the infinitesimal bimodules from \cite[\S3.1]{dtrpI}. Any graded operad  is a module over itself, and restriction of scalars is defined in the obvious way. 

The (aritywise) suspension $\Sigma M$ of an $\mathcal O$-module $M$ is again an $\mathcal O$-module with structure
\begin{align*}
(\sigma x)\circ_iy&=\sigma (x\circ_iy),&
y\circ_i(\sigma x)&=(-1)^{\abs{y}}\sigma(y\circ_ix),&
x\in M,\quad y\in\mathcal O.
\end{align*}

Suppose now that $\mathcal O$ is a relatively pseudo-cellular DG-operad of length $\alpha$, as in Definition \ref{rpc}. Recall that its underlying graded operad is $\mathcal O_0\amalg\mathcal F(V)$, $V=\bigoplus_{\beta<\alpha}V_{\beta}$. The sub-$\mathcal O_0$-module of $\mathcal O$ spanned by the identity element $\id{\mathcal O}$ is $\mathcal O_0$. The sub-$\mathcal O_0$-module $\langle V\rangle_{\mathcal O_0}\subset \mathcal O$ spanned by $V$ corresponds, in the direct sum decomposition for graded operad coproducts \eqref{coparb}, to the direct subsum indexed by the trees with exactly one inner vertex of even level, 
\[
\begin{tikzpicture}[level distance = 6mm, sibling distance = 4.8mm ]
\tikzstyle{level 3}=[level distance = 5mm, sibling distance = 5.5mm]
\tikzstyle{level 4}=[level distance = 4mm, sibling distance = 3.8mm]
\node [inner sep =0pt] {} [grow'=up]
child{[fill] circle (2pt)
	child{}
	child{node {$\cdots$} edge from parent [draw=none]}
	child{node {$\cdots$} edge from parent [draw=none]}
	child{[fill] circle (2pt)
		child{[fill] circle (2pt)
                  child{}
    child{node {$\cdots$} edge from parent [draw=none]}
    child{}}
		child{node {$\cdots$} edge from parent [draw=none]}
		child{[fill] circle (2pt)
child{}
    child{node {$\cdots$} edge from parent [draw=none]}
child{}}
}
	child{node {$\cdots$} edge from parent [draw=none]}
	child{}
	edge from parent [solid, thin]};
\end{tikzpicture}
\]
and it is freely generated by $V$, compare \cite[Proposition 18]{dtrpII}. Moreover, it clearly satisfies \[\Sigma \langle V\rangle_{\mathcal O_0}=\langle \Sigma V\rangle_{\mathcal O_0}.\]

We say that $\mathcal O$ is \emph{linear} if the restriction of the differential to the generators of the free part decomposes as
\[\xymatrix{
 V\ar[r]^-{\left(\begin{smallmatrix}
		d^0\\ d^1
		\end{smallmatrix}\right)}&\mathcal O_0\oplus \langle V\rangle_{\mathcal O_0}\subset \mathcal O.}\]
We call $d^{0}$ the \emph{constant part} of the differential. If $d^{0}=0$, we say that $\mathcal O$ is \emph{strictly linear}. 

\begin{prop}\label{cyllin}
	If $\mathcal O$ is a linear DG-operad as above, the canonical strong pseudo-cylinder $I\mathcal O$ has underlying graded operad $\mathcal O_0\amalg\mathcal F(IV)$, and the differential is defined on the free part by the following formulas, $x\in V$,
	\begin{align*}
		di_0(x)&=i_0d(x), &
		di_1(x)&=i_1d(x), &
		d(\sigma x)&=i_0(x)-i_1(x)-\sigma d^1(x).
	\end{align*}
\end{prop}


\begin{proof}
The proof of this proposition is intertwined with the proof of the following technical statement. 
We check, by induction on the length $\alpha$, that the chain homotopy $h_{\mathcal O}\colon I\mathcal O\r I\mathcal O$ (co)restricts to the following $\mathcal O_{0}$-module morphism of degree $+1$,
\[\xymatrix@C=10pt{\mathcal O_0\oplus\langle IV\rangle_{\mathcal O_{0}}
	=\mathcal O_0\oplus\langle V\rangle_{\mathcal O_{0}}\oplus \langle \Sigma V\rangle_{\mathcal O_{0}}\oplus \langle V\rangle_{\mathcal O_{0}}
	\ar[d]^-{\left(
		\begin{smallmatrix}
		0&0&0&0\\0&0&0&\sigma\\0&0&0&0
		\end{smallmatrix}
		\right)}\\ \;\langle IV\rangle_{\mathcal O_{0}}=\langle V\rangle_{\mathcal O_{0}}\oplus \langle \Sigma V\rangle_{\mathcal O_{0}}\oplus \langle V\rangle_{\mathcal O_{0}}.}\]
		
The case $\alpha=0$ is obvious since we are taking the trivial strong pseudo-cylinder on $\mathcal O_0$. If $\alpha$ is a limit ordinal, it follows by continuity.  Let $\alpha=\beta+1$ be a successor. The restriction of $h_{\mathcal O}$ to $\mathcal O_0$ coincides with $h_{\mathcal O_0}=0$. Hence, it suffices to prove that, for any 
$y\in IV_\beta(n)$, and $x_0,\dots, x_{n}\in \mathcal O_{0}$, 
\[h_{\mathcal O}\left(\begin{array}{c}\begin{tikzpicture}[level distance = 6mm, sibling distance = 4.8mm ]
\tikzstyle{level 3}=[level distance = 5mm, sibling distance = 5.5mm]
\tikzstyle{level 4}=[level distance = 4mm, sibling distance = 3.8mm]
\node [inner sep =0pt] {} [grow'=up]
child{[fill] circle (2pt)
	child{}
	child{node {$\cdots$} edge from parent [draw=none]}
	child{node {$\cdots$} edge from parent [draw=none]}
	child{[fill] circle (2pt)
		child{[fill] circle (2pt)
                  child{}
    child{node {$\cdots$} edge from parent [draw=none]}
    child{}
node [below left] {$\scriptstyle x_1$}}
		child{node {$\cdots$} edge from parent [draw=none]}
		child{[fill] circle (2pt)
child{}
    child{node {$\cdots$} edge from parent [draw=none]}
child{}
node [below right] {$\scriptstyle x_n$}}
		node [below left] {$\scriptstyle y$}}
	child{node {$\cdots$} edge from parent [draw=none]}
	child{}
	node [below right] {$\scriptstyle x_0$}
	edge from parent [solid, thin]};
\end{tikzpicture}\end{array}\right)=
(-1)^{\abs{x_0}}\begin{array}{c}\begin{tikzpicture}[level distance = 6mm, sibling distance = 4.8mm ]
\tikzstyle{level 3}=[level distance = 5mm, sibling distance = 5.5mm]
\tikzstyle{level 4}=[level distance = 4mm, sibling distance = 3.8mm]
\node [inner sep =0pt] {} [grow'=up]
child{[fill] circle (2pt)
  child{}
  child{node {$\cdots$} edge from parent [draw=none]}
  child{node {$\cdots$} edge from parent [draw=none]}
  child{[fill] circle (2pt)
    child{[fill] circle (2pt)
                  child{}
    child{node {$\cdots$} edge from parent [draw=none]}
    child{}
node [below left] {$\scriptstyle x_1$}}
    child{node {$\cdots$} edge from parent [draw=none]}
    child{[fill] circle (2pt)
child{}
    child{node {$\cdots$} edge from parent [draw=none]}
child{}
node [below right] {$\scriptstyle x_n$}}
    node [below left] {$\scriptscriptstyle h_{I}(y)$}}
  child{node {$\cdots$} edge from parent [draw=none]}
  child{}
  node [below right] {$\scriptstyle x_0$}
  edge from parent [solid, thin]};
\end{tikzpicture}\end{array}.\]

The element
\[
\begin{array}{c}
\begin{tikzpicture}[level distance = 6mm, sibling distance = 4.8mm ]
\tikzstyle{level 3}=[level distance = 5mm, sibling distance = 5.5mm]
\tikzstyle{level 4}=[level distance = 4mm, sibling distance = 3.8mm]
\node [inner sep =0pt] {} [grow'=up]
child{[fill] circle (2pt)
	child{}
	child{node {$\cdots$} edge from parent [draw=none]}
	child{node {$\cdots$} edge from parent [draw=none]}
	child{[fill] circle (2pt)
		child{[fill] circle (2pt)
                  child{}
    child{node {$\cdots$} edge from parent [draw=none]}
    child{}
node [below left] {$\scriptstyle x_1$}}
		child{node {$\cdots$} edge from parent [draw=none]}
		child{[fill] circle (2pt)
child{}
    child{node {$\cdots$} edge from parent [draw=none]}
child{}
node [below right] {$\scriptstyle x_n$}}
		node [below left] {$\scriptstyle y$}}
	child{node {$\cdots$} edge from parent [draw=none]}
	child{}
	node [below right] {$\scriptstyle x_0$}
	edge from parent [solid, thin]};
\end{tikzpicture}
\end{array}\in I\mathcal O=I\mathcal O_{\beta}\amalg_{\partial_{\beta,I}}\mathcal F(IV_{\beta})
\]
has filtration level $1$, with respect to the filtration in the proof of Theorem \ref{pasting2}. Therefore, by Remark  \ref{filtration},
\begin{align*}
h_{\mathcal O}\left(\begin{array}{c}\begin{tikzpicture}[level distance = 6mm, sibling distance = 4.8mm ]
\tikzstyle{level 3}=[level distance = 5mm, sibling distance = 5.5mm]
\tikzstyle{level 4}=[level distance = 4mm, sibling distance = 3.8mm]
\node [inner sep =0pt] {} [grow'=up]
child{[fill] circle (2pt)
	child{}
	child{node {$\cdots$} edge from parent [draw=none]}
	child{node {$\cdots$} edge from parent [draw=none]}
	child{[fill] circle (2pt)
		child{[fill] circle (2pt)
                  child{}
    child{node {$\cdots$} edge from parent [draw=none]}
    child{}
node [below left] {$\scriptstyle x_1$}}
		child{node {$\cdots$} edge from parent [draw=none]}
		child{[fill] circle (2pt)
child{}
    child{node {$\cdots$} edge from parent [draw=none]}
child{}
node [below right] {$\scriptstyle x_n$}}
		node [below left] {$\scriptstyle y$}}
	child{node {$\cdots$} edge from parent [draw=none]}
	child{}
	node [below right] {$\scriptstyle x_0$}
	edge from parent [solid, thin]};
\end{tikzpicture}\end{array}\right)={}&
h_{\mathcal O_{\beta},V_{\beta}}\left(\begin{array}{c}\begin{tikzpicture}[level distance = 6mm, sibling distance = 4.8mm ]
\tikzstyle{level 3}=[level distance = 5mm, sibling distance = 5.5mm]
\tikzstyle{level 4}=[level distance = 4mm, sibling distance = 3.8mm]
\node [inner sep =0pt] {} [grow'=up]
child{[fill] circle (2pt)
	child{}
	child{node {$\cdots$} edge from parent [draw=none]}
	child{node {$\cdots$} edge from parent [draw=none]}
	child{[fill] circle (2pt)
		child{[fill] circle (2pt)
                  child{}
    child{node {$\cdots$} edge from parent [draw=none]}
    child{}
node [below left] {$\scriptstyle x_1$}}
		child{node {$\cdots$} edge from parent [draw=none]}
		child{[fill] circle (2pt)
child{}
    child{node {$\cdots$} edge from parent [draw=none]}
child{}
node [below right] {$\scriptstyle x_n$}}
		node [below left] {$\scriptstyle y$}}
	child{node {$\cdots$} edge from parent [draw=none]}
	child{}
	node [below right] {$\scriptstyle x_0$}
	edge from parent [solid, thin]};
\end{tikzpicture}\end{array}\right)
\\&+
h_{\mathcal O_{\beta}}\partial_{\beta,I}h_{\mathcal O_{\beta},V_{\beta}}\left(\begin{array}{c}\begin{tikzpicture}[level distance = 6mm, sibling distance = 4.8mm ]
\tikzstyle{level 3}=[level distance = 5mm, sibling distance = 5.5mm]
\tikzstyle{level 4}=[level distance = 4mm, sibling distance = 3.8mm]
\node [inner sep =0pt] {} [grow'=up]
child{[fill] circle (2pt)
	child{}
	child{node {$\cdots$} edge from parent [draw=none]}
	child{node {$\cdots$} edge from parent [draw=none]}
	child{[fill] circle (2pt)
		child{[fill] circle (2pt)
                  child{}
    child{node {$\cdots$} edge from parent [draw=none]}
    child{}
node [below left] {$\scriptstyle x_1$}}
		child{node {$\cdots$} edge from parent [draw=none]}
		child{[fill] circle (2pt)
child{}
    child{node {$\cdots$} edge from parent [draw=none]}
child{}
node [below right] {$\scriptstyle x_n$}}
		node [below left] {$\scriptstyle y$}}
	child{node {$\cdots$} edge from parent [draw=none]}
	child{}
	node [below right] {$\scriptstyle x_0$}
	edge from parent [solid, thin]};
\end{tikzpicture}\end{array}\right).
\end{align*}
By Remark \ref{formulae}, since we are taking the trivial strong pseudo-cylinder on $\mathcal O_{0}$,
\[h_{\mathcal O_\beta,V_\beta}\left(\begin{array}{c}\begin{tikzpicture}[level distance = 6mm, sibling distance = 4.8mm ]
\tikzstyle{level 3}=[level distance = 5mm, sibling distance = 5.5mm]
\tikzstyle{level 4}=[level distance = 4mm, sibling distance = 3.8mm]
\node [inner sep =0pt] {} [grow'=up]
child{[fill] circle (2pt)
	child{}
	child{node {$\cdots$} edge from parent [draw=none]}
	child{node {$\cdots$} edge from parent [draw=none]}
	child{[fill] circle (2pt)
		child{[fill] circle (2pt)
                  child{}
    child{node {$\cdots$} edge from parent [draw=none]}
    child{}
node [below left] {$\scriptstyle x_1$}}
		child{node {$\cdots$} edge from parent [draw=none]}
		child{[fill] circle (2pt)
child{}
    child{node {$\cdots$} edge from parent [draw=none]}
child{}
node [below right] {$\scriptstyle x_n$}}
		node [below left] {$\scriptstyle y$}}
	child{node {$\cdots$} edge from parent [draw=none]}
	child{}
	node [below right] {$\scriptstyle x_0$}
	edge from parent [solid, thin]};
\end{tikzpicture}\end{array}\right)=(-1)^{\abs{x_0}}
\begin{array}{c}\begin{tikzpicture}[level distance = 6mm, sibling distance = 4.8mm ]
\tikzstyle{level 3}=[level distance = 5mm, sibling distance = 5.5mm]
\tikzstyle{level 4}=[level distance = 4mm, sibling distance = 3.8mm]
\node [inner sep =0pt] {} [grow'=up]
child{[fill] circle (2pt)
  child{}
  child{node {$\cdots$} edge from parent [draw=none]}
  child{node {$\cdots$} edge from parent [draw=none]}
  child{[fill] circle (2pt)
    child{[fill] circle (2pt)
                  child{}
    child{node {$\cdots$} edge from parent [draw=none]}
    child{}
node [below left] {$\scriptstyle x_1$}}
    child{node {$\cdots$} edge from parent [draw=none]}
    child{[fill] circle (2pt)
child{}
    child{node {$\cdots$} edge from parent [draw=none]}
child{}
node [below right] {$\scriptstyle x_n$}}
    node [below left] {$\scriptscriptstyle h_{I}(y)$}}
  child{node {$\cdots$} edge from parent [draw=none]}
  child{}
  node [below right] {$\scriptstyle x_0$}
  edge from parent [solid, thin]};
\end{tikzpicture}\end{array}.\]
Moreover, by Remark \ref{formulae2},
\[\partial_{\beta,I}
\left(\begin{array}{c}\begin{tikzpicture}[level distance = 6mm, sibling distance = 4.8mm ]
\tikzstyle{level 3}=[level distance = 5mm, sibling distance = 5.5mm]
\tikzstyle{level 4}=[level distance = 4mm, sibling distance = 3.8mm]
\node [inner sep =0pt] {} [grow'=up]
child{[fill] circle (2pt)
  child{}
  child{node {$\cdots$} edge from parent [draw=none]}
  child{node {$\cdots$} edge from parent [draw=none]}
  child{[fill] circle (2pt)
    child{[fill] circle (2pt)
                  child{}
    child{node {$\cdots$} edge from parent [draw=none]}
    child{}
node [below left] {$\scriptstyle x_1$}}
    child{node {$\cdots$} edge from parent [draw=none]}
    child{[fill] circle (2pt)
child{}
    child{node {$\cdots$} edge from parent [draw=none]}
child{}
node [below right] {$\scriptstyle x_n$}}
    node [below left] {$\scriptscriptstyle h_{I}(y)$}}
  child{node {$\cdots$} edge from parent [draw=none]}
  child{}
  node [below right] {$\scriptstyle x_0$}
  edge from parent [solid, thin]};
\end{tikzpicture}\end{array}\right)
=(-1)^{\abs{x_0}}
\begin{array}{c}\begin{tikzpicture}[level distance = 6mm, sibling distance = 4.8mm ]
\tikzstyle{level 3}=[level distance = 5mm, sibling distance = 5.5mm]
\tikzstyle{level 4}=[level distance = 4mm, sibling distance = 3.8mm]
\node [inner sep =0pt] {} [grow'=up]
child{[fill] circle (2pt)
  child{}
  child{node {} edge from parent [draw=none]}
  child{node {} edge from parent [draw=none]}
  child{[fill] circle (2pt)
    child{[fill] circle (2pt)
                  child{}
    child{node {$\cdots$} edge from parent [draw=none]}
    child{}
node [below left] {$\scriptstyle x_1$}}
    child{node {$\cdots$} edge from parent [draw=none]}
    child{[fill] circle (2pt)
child{}
    child{node {$\cdots$} edge from parent [draw=none]}
child{}
node [below right] {$\scriptstyle x_n$}}
    node [left] {$\scriptscriptstyle \partial_{\beta,I}h_{I}(y)$}}
  child{node {$\cdots$} edge from parent [draw=none]}
  child{}
  node [below right] {$\scriptstyle x_0$}
  edge from parent [solid, thin]};
\end{tikzpicture}\end{array}.\]

By linearity, $\partial_\beta$ factors as 
\[\xymatrix{
	\partial_\beta\colon V_\beta\ar[r]^-{\left(\begin{smallmatrix}
		\partial_\beta^0\\ \partial_\beta^1
		\end{smallmatrix}\right)}&\mathcal O_0\oplus \langle \bigoplus_{\gamma<\beta}V_\gamma\rangle_{\mathcal O_0}\subset \mathcal O_{\beta}.}\]
The maps $i_{0},i_{1}\colon \mathcal O_{\beta}\r I\mathcal O_{\beta}$ (co)restrict to the left and right vertical $\mathcal O_{0}$-module maps in the following diagram, respectively,
\[\xymatrix{
\mathcal O_0\oplus \langle \bigoplus_{\gamma<\beta}V_\gamma\rangle_{\mathcal O_0}
\ar@<-1ex>[d]_-{\left(
    \begin{smallmatrix}
    1&0\\0&1\\0&0\\0&0
    \end{smallmatrix}
    \right)}
	\ar@<1ex>[d]^-{\left(
		\begin{smallmatrix}
		1&0\\0&0\\0&0\\0&1
		\end{smallmatrix}
		\right)}
\\
\mathcal O_0
\oplus \langle \bigoplus_{\gamma<\beta}V_\gamma\rangle_{\mathcal O_0}
\oplus \langle \bigoplus_{\gamma<\beta}\Sigma V_\gamma\rangle_{\mathcal O_0}
\oplus \langle \bigoplus_{\gamma<\beta}V_\gamma\rangle_{\mathcal O_0}
}\]
Since $\partial_{\beta,I}=\left(\begin{smallmatrix}
	i_0\partial_{\beta}&-h_{\mathcal O_{\beta}}i_1\partial_{\beta}&i_1\partial_{\beta}
	\end{smallmatrix}\right)$, using the induction hypothesis on $h_{\mathcal O_{\beta}}$, we deduce that 
\[
h_{\mathcal O_{\beta}}
\left(
\begin{array}{c}\begin{tikzpicture}[level distance = 6mm, sibling distance = 4.8mm ]
\tikzstyle{level 3}=[level distance = 5mm, sibling distance = 5.5mm]
\tikzstyle{level 4}=[level distance = 4mm, sibling distance = 3.8mm]
\node [inner sep =0pt] {} [grow'=up]
child{[fill] circle (2pt)
  child{}
  child{node {} edge from parent [draw=none]}
  child{node {} edge from parent [draw=none]}
  child{[fill] circle (2pt)
    child{[fill] circle (2pt)
                  child{}
    child{node {$\cdots$} edge from parent [draw=none]}
    child{}
node [below left] {$\scriptstyle x_1$}}
    child{node {$\cdots$} edge from parent [draw=none]}
    child{[fill] circle (2pt)
child{}
    child{node {$\cdots$} edge from parent [draw=none]}
child{}
node [below right] {$\scriptstyle x_n$}}
    node [left] {$\scriptscriptstyle \partial_{\beta,I}h_{I}(y)$}}
  child{node {$\cdots$} edge from parent [draw=none]}
  child{}
  node [below right] {$\scriptstyle x_0$}
  edge from parent [solid, thin]};
\end{tikzpicture}\end{array}
\right)=(-1)^{\abs{x_0}}
\begin{array}{c}\begin{tikzpicture}[level distance = 6mm, sibling distance = 4.8mm ]
\tikzstyle{level 3}=[level distance = 5mm, sibling distance = 5.5mm]
\tikzstyle{level 4}=[level distance = 4mm, sibling distance = 3.8mm]
\node [inner sep =0pt] {} [grow'=up]
child{[fill] circle (2pt)
	child{}
	child{node {$\cdots$} edge from parent [draw=none]}
	child{node {$\cdots$} edge from parent [draw=none]}
	child{[fill] circle (2pt)
		child{[fill] circle (2pt)
                  child{}
    child{node {$\cdots$} edge from parent [draw=none]}
    child{}
node [below left] {$\scriptstyle x_1$}}
		child{node {$\cdots$} edge from parent [draw=none]}
		child{[fill] circle (2pt)
child{}
    child{node {$\cdots$} edge from parent [draw=none]}
child{}
node [below right] {$\scriptstyle x_n$}}
		node [below left] {$\scriptstyle \phi(y)$}}
	child{node {$\cdots$} edge from parent [draw=none]}
	child{}
	node [below right] {$\scriptstyle x_0$}
	edge from parent [solid, thin]};
\end{tikzpicture}\end{array},
\]
where $\phi=h_{\mathcal O_{\beta}}\partial_{\beta,I}h_{I}$ is the following product of (block) matrices,
\begin{align*}
\phi=
\left(
\begin{smallmatrix}
0&0&0&0\\
0&0&0&\sigma\\
0&0&0&0
\end{smallmatrix}
\right)
\left(\!\!\!
\begin{array}{c|c|c}
\left(
    \begin{smallmatrix}
    1&0\\0&1\\0&0\\0&0
    \end{smallmatrix}
    \right)
\left(\begin{smallmatrix}
		\partial_\beta^0\\ \partial_\beta^1
		\end{smallmatrix}\right)
		&
		\!\!\!\begin{array}{c}
		-
		\left(
\begin{smallmatrix}
0&0&0&0\\
0&0&0&\sigma\\
0&0&0&0
\end{smallmatrix}
\right)
		\left(
		\begin{smallmatrix}
		1&0\\0&0\\0&0\\0&1
		\end{smallmatrix}
		\right)
		\left(\begin{smallmatrix}
		\partial_\beta^0\\ \partial_\beta^1
		\end{smallmatrix}\right)\\[3mm]
		\hline 0
		\end{array}\!\!\!
		&
		\left(
		\begin{smallmatrix}
		1&0\\0&0\\0&0\\0&1
		\end{smallmatrix}
		\right)
		\left(\begin{smallmatrix}
		\partial_\beta^0\\ \partial_\beta^1
		\end{smallmatrix}\right)
\end{array}\!\!\!
\right)
\left(
\begin{smallmatrix}
0&0&0\\
0&0&\sigma\\
0&0&0
\end{smallmatrix}
\right).
\end{align*}
A straightforward computation shows that $\phi=0$, hence the previous claim follows.

As part of the proof of the previous statement, we have shown that, for any linear DG-operad $\mathcal O$ of any length $\alpha$, and for any $\beta<\alpha$, $h_{\mathcal O_{\beta}}i_{1}\partial_{\beta}\colon V_{\beta}\r I\mathcal O_{\beta}$ corestricts to
\[
\left(
\begin{smallmatrix}
0&0&0&0\\
0&0&0&\sigma\\
0&0&0&0
\end{smallmatrix}
\right)
		\left(
		\begin{smallmatrix}
		1&0\\0&0\\0&0\\0&1
		\end{smallmatrix}
		\right)
		\left(\begin{smallmatrix}
		\partial_\beta^0\\ \partial_\beta^1
		\end{smallmatrix}\right)=
		\left(
\begin{smallmatrix}
0\\
\sigma\partial_{\beta}^{1}\\
0
\end{smallmatrix}
\right)\colon V_{\beta}\To \langle \bigoplus_{\gamma<\beta}V_\gamma\rangle_{\mathcal O_0}
\oplus \langle \bigoplus_{\gamma<\beta}\Sigma V_\gamma\rangle_{\mathcal O_0}
\oplus \langle \bigoplus_{\gamma<\beta}V_\gamma\rangle_{\mathcal O_0}.
\]
The statement of the proposition is equivalent to this, see \eqref{dsigma2}.
\end{proof}

\begin{exm}
	Let $\mathcal O$ be the operad obtained by quotienting out $\mu_n$, $n\geq 3$, from $\mathcal A_\infty^D$ in Example \ref{conder}. Its algebras are non-unital DG-algebras with an up-to-homotopy derivation. As a graded operad, $\mathcal O$ is generated by
	$$\mu_{2}\in\mathcal{O}_{\infty}(2)_{0},
	\qquad D_{n}\in\mathcal{O}_{\infty}(n)_{n-1},\quad n\geq 1,$$
	with a single relation \[\mu_2\circ_1\mu_2=\mu_2\circ_2\mu_2,\]
	and differential
	\begin{align*}
	d(\mu_2)={}&0,&
	d(D_n)={}&\mu_2\circ_1D_{n-1}+(-1)^n\mu_2\circ_2D_{n-1}+\sum_{i=1}^{n-1}(-1)^{n+i}D_{n-1}\circ_i\mu_2.
	\end{align*}
	The operadic suspension $\Lambda\mathcal O$ is therefore generated by 
	$$\mu_{2}\in\Lambda\mathcal{O}_{\infty}(2)_{-1},
	\qquad D_{n}\in\Lambda \mathcal{O}_{\infty}(n)_{0},\quad n\geq 1,$$
	with a single relation, the same as above, 
	and differential
	\begin{align*}
	d(\mu_2)={}&0,&
	d(D_n)={}&\mu_2\{D_{n-1}\}-D_{n-1}\{\mu_2\}.
	\end{align*}
	We regard $\Lambda \mathcal O$ as a relatively cellular DG-operad with $\mathcal O_0$ the associative DG-operad, i.e.~$\Lambda\mathcal O_0$ is the sub-DG-operad generated by $\mu_2$, 
	and \[\Lambda \mathcal O_n=\Lambda \mathcal O_{n-1}\amalg_{\partial_{n-1}}\mathcal F(\Bbbk\cdot D_{n}),\qquad \partial_{n-1}(D_{n})=d(D_{n}).\]
	Then it is strictly linear with
	\[\quad d^1(D_n)=d(D_n),\quad n\geq 1.\]
	Therefore, Proposition \ref{cyllin} yields
	\begin{align*}
	d(\sigma D_n)&=i_0(D_n)-i_1(D_n)-\sigma d(D_n)\\
	&=i_0(D_n)-i_1(D_n)-\sigma(\mu_2\{D_{n-1}\}-D_{n-1}\{\mu_2\})\\
	&=i_0(D_n)-i_1(D_n)+\mu_2\{\sigma D_{n-1}\}+\sigma D_{n-1}\{\mu_2\}.
	\end{align*}
	It is easy to check that this is indeed the formula obtained from \eqref{dsigma2} and the formula for $h_{\Lambda\mathcal A_\infty^D}i_1d(D_n)$ in Example \ref{conder}, killing $\sigma\mu_n$ for $n\geq 2$, $i_0\mu_n,i_1\mu_n$ for $n\geq 3$, and identifying $i_0\mu_2=i_1\mu_2=\mu_2$.
\end{exm}

In the following example, a non-strict linear DG-operad is also considered.

\begin{exm}\label{unital}
	The following linear relatively pseudo-cellular DG-operad $\mathcal O$ appears in \cite{udga}. Fix some $m>0$. As a graded operad, $\mathcal O$ is generated by
	\begin{align*}
		u&\in\mathcal O(0)_0,&
		\mu&\in\mathcal O(2)_0,&
		\nu_n^S&\in\mathcal O(n-m)_{n-2+m},
	\end{align*}
	where $n\geq m$ and $S\subset\{1,\dots,n\}$ runs over all subsets of cardinality $m$, with relations
	\begin{align*}
		\mu\circ_1\mu&=\mu\circ_2\mu,&
		\mu\circ_1u&=\id{}=\mu\circ_2u.
	\end{align*}
	The differential is defined by
	\begin{align*}
	d(u)&=0,&
	d(\mu)&=0;
	\end{align*}
	if $(n,m)\neq (2,1), (1,1)$,
	\begin{align*}
	d(\nu_n^S)={}&(-1)^{n}\mu\circ_1\nu_{n-1}^{S}\quad\text{unless }l_m=n\\
	&+\mu\circ_2\nu_{n-1}^{S-1}\quad\text{unless }l_1=1\\
	&+\hspace{-20pt}\sum_{
		\begin{array}{c}\\[-17pt]
		\scriptstyle 1\leq v\leq m+1\\[-5pt]
		\scriptstyle l_{v-1}< i+v-1< l_{v}-1
		\end{array}
	}
	\hspace{-20pt}
	(-1)^{i+v-1}\nu_{n-1}^{S_{v}\cup(S_{v}'-1)}\circ_i\mu;
	\end{align*}
	and if $m=1$ also
	\begin{align*}
	d(\nu_1^{\{1\}})&=0,&
	d(\nu_2^{\{1\}})&=\mu\circ_1\nu_1^{\{1\}}-\id{},&
	d(\nu_2^{\{2\}})&=\mu\circ_2\nu_1^{\{1\}}-\id{}.
	\end{align*}
	Here we denote  
 $S=\{l_{1},\dots,l_{|S|}\}$,  $l_0=0$, $l_{|S|+1}=n+1$,
 $S+t=\{l_{1}+t,\dots,l_{|S|}+t\}$, and
\begin{align*}
S_{v}&=\{l_{1},\dots,l_{v-1}\},&S'_{v}=&S\setminus S_{v}=\{l_{v},\dots,l_{|S|}\},&1\leq v\leq |S|+1.
\end{align*}
Unlike in previous cases, operadic suspension and braces do not simplify the definition of $\mathcal O$.
	
	This DG-operad is relatively cellular with $\mathcal O_0=\cdots=\mathcal O_{m-1}$ the unital associative operad, i.e.~the suboperad generated by $u$ and $\mu$, and
	\[\mathcal O_n=\mathcal O_{n-1}\amalg_{\partial_{n-1}}\mathcal F(\Bbbk\cdot\{\nu_n^S\}_{\substack{S\subset\{1,\dots,n\}\\\hspace{-18pt}\abs{S}=m}}),\qquad n\geq m.\]
	Here $\partial_{n-1}$ is defined as the differential above. Note that $\mathcal O$ is clearly linear, even strictly for $m>1$, but not for $m=1$, since the constant part of the differential satisfies \[d^{0}(\nu_{2}^{\{1\}})=d^{0}(\nu_{2}^{\{2\}})=-\id{}.\]

	Consider the retraction $r\colon\mathcal O\r\mathcal O_0$ defined as follows,
	\begin{align*}
		r(\nu_n^S)&=0\text{ if }(n,m)\neq(1,1),&
		r(\nu_1^{\{1\}})&\;\mapsto\;u\text{ if }m=1.
	\end{align*}
	Compatibility with differentials is checked in \cite[Lemma 5.9]{udga}. Denote by $j\colon\mathcal O_0\r\mathcal O$ the inclusion. The equations in the proof of \cite[Lemma 5.9]{udga} show that there is a homotopy 
	$H\colon I\mathcal O\r\mathcal O$,
	relative to $\mathcal O_0$, from the identity in $\mathcal O$ to $jr$, defined by
	\begin{align*}
		H(i_{0}\nu_{n}^{S})&=\nu_{n}^{S},&
		H(\sigma\nu_n^S)&=(-1)^{l_1+1}\nu_{n+1}^{S+1}\circ_{l_1}u,&
		H(i_{1}\nu_{n}^{S})&=jr(\nu_{n}^{S}).
	\end{align*}
	Therefore $j$ is the inclusion of a strong deformation retract, in particular a homotopy equivalence.
\end{exm}

We finally compute some structure maps on canonical strong pseudo-cylinders of linear DG-operads.

\begin{defn}\label{doubrev}
	Let $\mathcal O$ be a relatively pseudo-cellular DG-operad. 
	Consider the pasting of two canonical strong pseudo-cylinders $I\mathcal O {}_{i_1}\!{\cup}_{i_0}I\mathcal O$ obtained by identifying the top  copy of $\mathcal O$ ($i_1\mathcal O$) in the first $I\mathcal O$ with the bottom copy of $\mathcal O$ ($i_0\mathcal O$) in the second $I\mathcal O$. A \emph{doubling map} is a map \[\nu\colon I\mathcal O\To I\mathcal O{}_{i_1}\!\!\cup_{i_0}I\mathcal O,\]
	compatible with the projections onto $\mathcal O$, sending the bottom (resp.~top) copy of $\mathcal O$ in the source to the bottom (resp.~top) copy of $\mathcal O$ in the first (resp.~second) $I\mathcal O$ in the target. A \emph{reversing map} is a map
	\[\iota\colon I\mathcal O\To I\mathcal O\]
	compatible with the projections onto $\mathcal O$, sending the bottom (resp.~top) copy of $\mathcal O$ in the source to the top (resp.~bottom) copy of $\mathcal O$ in the target.
\end{defn}
	
	Doubling and reversing maps allow to vertically compose and invert homotopies. They are important since, if $\mathcal O$ is \emph{based}, i.e.~equipped with a retraction $\mathcal O\r\mathcal O_0$, they give rise to the up-to-homotopy cogroup structure of the model theoretic (relative) suspension $\Sigma\mathcal O$ of $\mathcal O$ and to the coaction of $\Sigma\mathcal O$ on the relative cone $C\mathcal O$. Using the chain homotopy in $I\mathcal O$, we could give formulas for doubling and reversing maps in all cases. Formulas are not easy in general, so we will content ourserlves with the linear case. 
	
	\begin{prop}\label{doubinvlin}
		Let us place ourselves in the context of Proposition \ref{cyllin}. Denote by $j_0,j_1\colon I\mathcal O\r I\mathcal O{}_{i_1}\!\cup_{i_0}I\mathcal O$ the inclusion of the first and second factor, respectively, which satisfy $j_0i_1=j_1i_0$. The following formulas define a doubling map and a reversing map in the sense of Definition \ref{doubrev}, $x\in V$,
			\begin{align*}
			\nu i_0(x)&=j_0i_0(x),&
			\nu i_1(x)&=j_1i_1(x),&
			\nu(\sigma x)&=j_0\sigma x+j_1\sigma x,\\
				\iota i_0(x)&=i_1(x),&
				\iota i_1(x)&=i_0(x),&
				\iota (\sigma x)&=-\sigma x.
			\end{align*}
	\end{prop}
	
	\begin{proof}
		The conditions for $\nu$ to be a doubling map are $(p,p)\nu=p$, $\nu i_0=j_0i_0$, and $\nu i_1=j_1i_1$. The only part about $\nu$ which is not completely trivial is compatibility with differentials in the third case. In order to check this, note that the third formula is actually true for any $x\in\langle V\rangle_{\mathcal O_0}$, since $i_0, i_1,j_0,j_1$ are maps relative to $\mathcal O_0$. Therefore,
		\begin{align*}
		d\nu(\sigma x)={}&dj_0\sigma x+dj_1\sigma x\\
		={}&j_0d(\sigma x)+j_1d(\sigma x)\\
		={}&j_0i_0(x)-j_0i_1(x)-j_0\sigma d^1(x)\\
		&+j_1i_0(x)-j_1i_1(x)-j_1\sigma d^1(x)\\
		={}&j_0i_0(x)-j_1i_1(x)-j_0\sigma d^1(x)-j_1\sigma d^1(x)\\
		={}&\nu i_0(x)-\nu i_1(x)-\nu \sigma d^1(x)\\
		={}&\nu d(\sigma x).
		\end{align*}
		One can similarly check that $\iota$ is a reversing map.
	\end{proof}

\providecommand{\bysame}{\leavevmode\hbox to3em{\hrulefill}\thinspace}
\providecommand{\MR}{\relax\ifhmode\unskip\space\fi MR }
\providecommand{\MRhref}[2]{%
	\href{http://www.ams.org/mathscinet-getitem?mr=#1}{#2}
}
\providecommand{\href}[2]{#2}

\end{document}